\documentclass[reqno,11pt]{amsproc}
\usepackage{amssymb}
\usepackage{amsmath}
\usepackage{amsfonts}
\usepackage{microtype}
\usepackage{xcolor}
\usepackage{geometry}
\usepackage{mathrsfs}
\usepackage{mathtools}

\usepackage[backend=biber,style=ieee,sorting=none,backref=true,citestyle=numeric-comp,firstinits=false]{biblatex}
\usepackage{tikz,tikz-cd}
\usetikzlibrary{arrows,cd,decorations.markings,arrows.meta}
\tikzset{->-/.style={decoration={markings,mark=at position .5 with {\arrow{>}}},postaction={decorate}}}
\tikzset{-<-/.style={decoration={markings,mark=at position .5 with {\arrow{<}}},postaction={decorate}}}
\tikzstyle{box}=[fill=white, draw=black, shape=rectangle, inner sep=14pt]
\tikzstyle{forward arrow}=[->-]
\tikzstyle{backward arrow}=[-<-]

\usepackage{tikzit}
\tikzstyle{box}=[fill=white, draw=black, shape=rectangle]

\tikzstyle{edge}=[-, draw=black]
\tikzstyle{arrow}=[->, draw=black]

\definecolor{myurlcolor}{rgb}{0,0,0.4}
\definecolor{mycitecolor}{rgb}{0,0.5,0}
\definecolor{myrefcolor}{rgb}{0.5,0,0}
\usepackage[breaklinks=true]{hyperref}
\hypersetup{colorlinks,
linkcolor=myrefcolor,
citecolor=mycitecolor,
urlcolor=myurlcolor}
\usepackage[capitalize]{cleveref}

\newcommand{\beq}{\begin{equation}}
\newcommand{\eeq}{\end{equation}}
\newcommand{\N}{\mathbb{N}}
\newcommand{\Z}{\mathbb{Z}}
\newcommand{\Q}{\mathbb{Q}}
\newcommand{\R}{\mathbb{R}}
\newcommand{\C}{\mathbb{C}}
\newcommand{\D}{\mathcal{D}}
\newcommand{\E}{\mathcal{E}}
\newcommand{\Der}{\mathrm{Der}}
\newcommand{\T}{\mathcal{T}}
\renewcommand{\L}{\mathcal{L}}
\newcommand{\op}{\mathrm{op}}
\newcommand{\id}{\mathsf{id}}		% identity morphism
\newcommand{\ZLA}{\mathcal{A}}		% algebraifold
\newcommand{\ZLB}{\mathcal{B}}		% Another algebraifold
\newcommand{\ZLC}{\mathcal{C}}		% Another algebraifold
\newcommand{\ZLL}{\mathcal{L}}		% Another algebraifold
\newcommand{\Ric}{\mathsf{Ric}}		% Ricci tensor
\newcommand{\Mod}{\mathcal{M}}		% module
\newcommand{\Nod}{\mathcal{N}}		% another module
\newcommand{\Modules}[1]{\mathrm{Hom}_{#1}}
\newcommand{\longmapsfrom}{\mathrel{\text{\rotatebox[origin=c]{180}{$\longmapsto$}}}}
\newcommand{\Man}{\mathsf{Man}}
\newcommand{\Afd}[1]{\mathsf{Afd}_{#1}}
\newcommand{\Alg}[1]{\mathsf{cAlg}_{#1}}
\newcommand{\Modcat}[1]{\mathsf{Mod}_{#1}}
\newcommand{\fgpModcat}[1]{\mathsf{fgpMod}_{#1}}
\newcommand{\Vect}[1]{\mathsf{Vect}_{#1}}
\newcommand{\eps}{\varepsilon}

\theoremstyle{plain}
\newtheorem{dummy}{Dummy}[section]
\newtheorem{thm}[dummy]{Theorem}\Crefname{thm}{Theorem}{Theorems}
\newtheorem{lem}[dummy]{Lemma}\Crefname{lem}{Lemma}{Lemmas}
\newtheorem{prop}[dummy]{Proposition}\Crefname{prop}{Proposition}{Propositions}
\Crefname{cor}{Corollary}{Corollaries}
\newtheorem{ass}[dummy]{Assumption}\Crefname{ass}{Assumption}{Assumptions}
\newtheorem{qstn}[dummy]{Question}\Crefname{qstn}{Question}{Questions}
\newtheorem{defn}[dummy]{Definition}\Crefname{defn}{Definition}{Definitions}
\newtheorem{nota}[dummy]{Notation}\Crefname{nota}{Notation}{Notations}
\newtheorem{prob}[dummy]{Problem}\Crefname{prob}{Problem}{Problems}
\newtheorem{slog}[dummy]{Slogan}\Crefname{slog}{Slogan}{Slogans}
\theoremstyle{remark}
\newtheorem{ex}[dummy]{Example}\Crefname{ex}{Example}{Examples}
\newtheorem{conj}[dummy]{Conjecture}\Crefname{conj}{Conjecture}{Conjectures}
\newtheorem{rem}[dummy]{Remark}\Crefname{rem}{Remark}{Remarks}
\newtheorem{note}[dummy]{Note}\Crefname{note}{Note}{Notes}
\numberwithin{equation}{section}
\Crefname{equation}{}{}		% abbreviate 'Equation (3.2)' to '(3.2)'

\Crefname{scalar_ex}{Scalar Field}{Scalar Field}
\Crefname{electro_ex}{Electrodynamic Field}{Electrodynamic Field}
\Crefname{spinor_electro_ex}{Spinor Field}{Spinor Field}
\Crefname{spinor_electro_ex}{Spinor Electrodynamics}{Spinor Electrodynamics}
\Crefname{gr_ex}{Metric Field}{Metric Field}
\Crefname{scalar_fun_ex}{Scalar Field (Kinematic)}{Scalar Field (Kinematic)}

\allowdisplaybreaks

\let\originalleft\left
\let\originalright\right
\renewcommand{\left}{\mathopen{}\mathclose\bgroup\originalleft}
\renewcommand{\right}{\aftergroup\egroup\originalright}

\usepackage{enumitem}
\setlist[enumerate]{label=(\roman*),itemsep=5pt,topsep=8pt}
\setlist[itemize]{label=$\triangleright$,itemsep=5pt,topsep=6pt}
\Crefformat{enumi}{#2#1#3}

\newcommand{\newterm}[1]{\textbf{#1}}

\usepackage{hyphenat}
\usepackage{todonotes}

\begin{document}

%-------------------------------------------------------------------

%%%%%%%%%%%% title page stuff %%%%%%%%%%%%%%%%%%%%%%%%%%

\title{Differential geometry and general relativity\\[2pt] with algebraifolds}

\author{Tobias Fritz}

\address{Department of Mathematics, University of Innsbruck, Austria}
\email{tobias.fritz@uibk.ac.at}



\thanks{\textit{Acknowledgements.} We thank Lu Chen, Igor Khavkine, Hông Vân Lê, Eugene Lerman and Michael Oberguggenberger for fruitful discussions, pointers to the literature and/or comments on a draft.}

\begin{abstract}
	It is often noted that many of the basic concepts of differential geometry, such as the definition of connection, are purely algebraic in nature.
	Here, we review and extend existing work on fully algebraic formulations of differential geometry which eliminate the need for an underlying manifold.
	While the literature contains various independent approaches to this, we focus on one particular approach that we argue to be the most natural one based on the definition of \emph{algebraifold}, by which we mean a commutative algebra $\ZLA$ for which the module of derivations of $\ZLA$ is finitely generated projective.
	Over $\R$ as the base ring, this class of algebras includes the algebra $C^\infty(M)$ of smooth functions on a manifold $M$, and similarly for analytic functions.
	An importantly different example is the Colombeau algebra of generalized functions on $M$, which makes distributional differential geometry an instance of our formalism.
	Another instance is a fibred version of smooth differential geometry, since any smooth submersion $M \to N$ makes $C^\infty(M)$ into an algebraifold with $C^\infty(N)$ as the base ring.
	Over any field $k$ of characteristic zero, examples include the algebra of regular functions on a smooth affine variety as well as any function field.

	Our development of differential geometry in terms of algebraifolds comprises tensors, connections, curvature, geodesics and we briefly consider general relativity.
\end{abstract}

\newgeometry{top=2cm}
\maketitle
\tableofcontents

\newpage
\restoregeometry

\section{Introduction}
\label{introduction}

It is a standard observation that differential geometry can largely be formulated in algebraic terms. This manifests itself for example in the definition of a vector field as a derivation on the algebra of smooth functions $C^\infty(M)$; and more generally in the characterization of a tensor as a map enjoying $C^\infty(M)$-linearity in each argument; but also in the definition of a connection as an operation of vector fields satisfying suitable algebraic conditions. It is therefore not surprising that pseudo-Riemannian geometry, up to the consideration of metric tensors and the Einstein field equation of general relativity, can be formulated entirely in terms of constructions making reference only to the algebra $C^\infty(M)$ without making reference to the points of the manifold $M$.
It is perhaps also not surprising that these algebraic constructions only rely on certain properties of $C^\infty(M)$ shared by many other algebras that are not of this form, to which the same constructions can be applied. Such an algebraic treatment can thus provide not just a reformulation of basic differential geometry without reference to points, but at the same time will amount to a considerable generalization. In the context of general relativity, such a version of algebraic differential geometry has first been proposed by Geroch~\cite{geroch}.
There are a number of further developments and apparent independent rediscoveries of this idea~\cite{heller_differential,heller_einstein,heller_sheaves,malliosI,malliosII,mallios,MR_einstein,schmidt,BM,Pessers,PV}, which we turn to in \Cref{related_work}.

In this paper, we develop one such approach that seems particularly natural to us. 
The main observation behind the definitions in \Cref{sec:algebraifolds} is that if one starts with any commutative ring $k$ and any commutative $k$-algebra $\ZLA$ satisfying a suitable regularity condition, namely that its module of derivations is \emph{finitely generated projective}, then the standard algebraic definitions of tensors, connections and curvature still make perfect sense and can be put to use.
We call such algebras \newterm{algebraifolds} (\Cref{algebraifold}).
Among the basic examples of algebraifolds are of course the $\R$-algebras of smooth functions $C^\infty(M)$ for any smooth manifold $M$, and this construction results in an embedding of the category of smooth manifolds into the category of algebraifolds~\eqref{man_to_afd}.
The reason for why the algebraic versions of the standard definitions still apply to algebraifolds is that a finitely generated projective module is a \emph{dualizable object} in the category of modules; this is known as the \emph{dual basis lemma}, concepts that we recall in \Cref{fgp}.
Dual bases for algebraifolds (\Cref{dual_bases_algebraifold}) will turn out to plays a role similar to the role of charts for analysis on manifolds.

In \Cref{algebraifold_examples}, we focus on other examples of algebraifolds.
A simple example that makes sense over any commutative ring $k$ is the polynomial ring $k[x_1,\ldots,x_n]$, and we thus obtain a generalization of standard differential geometry that frees it from its confinement to the manifold context.
Other examples include algebras of analytic functions, all finitely generated separable field extensions (which illustrate that differential geometry can even be done with \emph{fields} of functions).
In addition, the (special) Colombeau algebra of generalized functions on a smooth manifold is an algebraifold, and this makes the corresponding version of distributional differential geometry an instance of our formalism.
As another interesting example, let us mention that for any smooth submersion $M \to N$, the ring $C^\infty(M)$ considered as an algebra over $k = C^\infty(N)$, in which case our formalism instantiates to a \emph{fibred} version of differential geometry which automatically keeps track of how all structures vary smoothly over $N$.

\Cref{sec:diff_geom} presents the notions of tensors, connections and curvature for algebraifolds.
Since these notions are already largely of algebraic nature in the manifold formalism, these developments are rather straightforward.
The only additional assumption that we need to make, in order for the Levi--Civita connection to exist, is that $2 \in k$ should be invertible (\Cref{two_inv}).
It may also be worth mentioning that our notion of dimension for algebraifolds (\Cref{dimension}) generalizes the dimension of a connected manifold, but is not an integer in general.

In \Cref{cats}, we consider morphisms of algebraifolds, which are algebra homomorphisms subject to an additional regularity condition (\Cref{algebraifold_homomorphism}), which may or may hold for every homomorphism (\Cref{existence_pullback}).
We then consider a notion of differential for such a map, which in the manifold case recovers the usual action of a smooth map on tangent vectors (\Cref{tangent_action}).
Subsequently, we consider the category of algebraifolds over any base $k$, with an algebraifold map defined as the formal opposite of an algebraifold homomorphism.
The \emph{problem of products} asks whether this category of algebraifolds has products, and whether these products specialize to products of smooth manifolds in the case $k = \R$ (\Cref{prods}).
A positive answer would be an important indication of the naturality of the algebraifold formalism.

We return to differential geometry in \Cref{geodesics} with a treatment of geodesics, which seems an important step in light of the fact that the standard manifold definition is strongly point-based.
This definition relies on the notion of differential developed in \Cref{cats} as well as on a notion of \emph{formal line}.
The idea is that algebraifold maps from a formal line into another algebraifold generalize the notion of smooth curve as a smooth map from $\R$ into a smooth manifold.

Finally, we sketch potential applications of algebraifolds to general relativity in \Cref{GR}.
The first one is a cosmological solution to the Einstein field equations given by a rational function field that works over any ground field $k$.
The appeal of this solution is that it is a field, and this realizes the non-rigorous physicist's dream of being able to invert any nonzero function on spacetime without worrying about its invertibility.
We also consider the dependence of a (smooth) cosmological spacetime on parameters, which is naturally described in our formalism as an algebraifold over the ring of smooth functions of the parameters.

\subsection{Conventions}

All our rings and algebras are assumed to have a unit $1$ and are commutative.
Except where explicitly noted otherwise, $k$ is any commutative ring playing the role of base ring.

\section{Related work}
\label{related_work}`

Algebraic approaches to differential calculus have been quite standard for a long time~\cite{sardanashvily}.
In particular they have been considered to have relevance for partial differential equations in physics by Vinogradov's school~\cite{vinogradov}.
For a somewhat different context, similar ideas play a central role in noncommutative geometry~\cite[Chapters 10 and 11]{DV}.
These pointers should suffice to indicate that the literature on algebraic approaches to differential geometry in general is much too broad for us to provide a comprehensive overview here.
We therefore limit our discussion to those works which have also provided an explicit treatment of metrics, connections and the curvature tensor.
Then we know of the following, in roughly chronological order:

\subsection{Geroch}

The first such approach that we are aware of is Geroch's~\cite{geroch}, who had considered an algebraic version (over $\R$) of metrics, tensor calculus, connections and curvature.
This allows one to consider the Einstein field equation in algebraic terms, and Geroch coined the term \newterm{Einstein algebra} for its solutions.
In addition  to serving as inspiration for some of the works mentioned next, this proposal seems to have found the most resonance in the philosophy of physics literature~\cite{earman,bain,rynasiewicz,RBW}.

\subsection{Heller}

There is a good amount of work by Heller and coauthors which builds on Geroch's idea and extends it in various directions, in particular extending it to \emph{sheaves} of Einstein algebras~\cite{heller_differential,heller_einstein,heller_sheaves,HS}.
This is motivated by the desire to incorporate spacetime singularities as actual points, which is achieved by imposing the differential-geometric structure and the Einstein field equation only on suitable open subsets which do not contain any singular points.
Heller's work places particular emphasis on Gelfand-type representations, meaning that considering abstract algebras as concrete algebras of functions.

\subsection{Mallios}

The \emph{abstract differential geometry} of Mallios and coauthors~\cite{MR,malliosI,malliosII,mallios,MR_einstein} also uses sheaves of algebras $\ZLA$, but now equipped with a sheaf of modules $\Omega^1$ playing the role of $1$-forms and a derivation $d$ taking values in $\Omega^1$.
Mallios calls these \emph{differential triads}.
An additional structure often considered is a sheaf morphism $d^1 : \Omega^1 \longrightarrow \Omega^1 \wedge_\ZLA \Omega^1$, out of which a de Rham complex and de Rham cohomology can be built~\cite{MR}.
Notions of metric, connection and curvature are developed in~\cite{malliosII}.

\subsection{Schmidt} 

The thesis of Schmidt~\cite{schmidt} has formulated pseudo-Riemannian geometry and general relativity in scheme-theoretic terms based on Kähler differentials, up to the construction of various scheme-theoretic solutions to the Einstein field equations.
This approach does not specialize to manifold differential geometry, but it is argued that cosmological solutions to the Einstein field equation are algebraic spaces (in the algebraic geometry sense), which makes this formulation of differential geometry physically adequate.

However, Schmidt's emphasis is on \emph{arithmetic} geometry, and in particular on spacetimes defined over the integers or the ring of \emph{adeles}.
The latter means that all possible completions of the rational numbers are treated on an equal footing, these completions being the real numbers on the one hand and the $p$-adic numbers for all primes $p \in \N$ on the other.
%Manin suggests that the right type of algebraic structure to do physics with incorporates both real numbers and all $p$-adic numbers, and this is exactly what the ring of adeles does.
Working over the ring of adeles in this way is in the spirit of Manin's hypothesized \emph{arithmetical physics}~\cite{manin}, who had already advocated the idea that all completions of the rational numbers should be treated democratically. 

\subsection{Beggs and Majid}

The \emph{quantum Riemannian geometry} of Beggs and Majid~\cite{BM}\footnote{This research program features a substantial number of papers now, we thus refer to their textbook~\cite{BM} for further references.}
	develops a differential calculus for a noncommutative version of differential geometry, where the basic structure considered is an algebra of functions $\ZLA$ together with a bimodule of $1$-forms $\Omega^1$ equipped with a differential $d : \ZLA \to \Omega^1$ together with additional technical conditions.
This is quite similar to Mallios's approach, but differs insofar as the left and right actions of $\ZLA$ on $\Omega^1$ are not necessarily equal even when $\ZLA$ is commutative, and Beggs and Majid present examples where this seems natural~\cite[Example~1.6]{BM}.
They also note that vector fields, and derived notions such as the Levi--Civita connection, can be meaningfully considered as soon as $\Omega^1$ is finitely generated projective.
Their most recent developments include a theory of quantum geodesic flows and their relation to quantum mechanics~\cite{BMflows,BM0923,BM1223}

The focus of Beggs and Majid is clearly in a noncommutative context.
This is motivated by quantum gravity, and they correspondingly investigate noncommutative examples of their formalism as models of quantum spacetime.

\subsection{Pessers and van der Veken}

Apparently also unaware of the earlier works, the thesis of Pessers~\cite{Pessers} and the subsequent paper by Pessers and van der Veken~\cite{PV} consider an approach based on Lie--Rinehart algebras, which are given by an algebra $\ZLA$ together with a module $\D$ which acts on $\ZLA$ by derivations~\cite{huebschmann}.
They call such a structures \emph{Rinehart space} provided that a few technical conditions are satisfies, 
Roughly speaking, these Rinehart spaces correspond to Mallios's differential triads modulo the dualization between $1$-forms (as in differential triads) and vector fields (as in Rinehart spaces).

Pessers and van der Veken subsequently consider metrics, connections, the Levi--Civita connection and curvature, before moving on to a developments on quotients of Rinehart spaces, which constitute an algebraic formulation and generalization of submanifolds, and Rinehart spaces of constant sectional curvature.

\subsection{Novelty in our approach}

The particular approach that we put forward in this paper is distinct from all the ones above, but still similar in flavour to the extent that the basic structures of differential geometry (\Cref{sec:diff_geom}) are treated in essentially the same way, 
As such, our first goal is to serve as a review of algebraic differential geometry in a setting that seems particularly simple.
In addition to this, the following original points may be worth highlighting separately:

\begin{itemize}
	\item While the importance of finitely generated projective modules for the duality between vector fields and $1$-forms has also been recognized e.g.~by Beggs and Majid, our definition of algebraifold (\Cref{algebraifold}) seems to be new.
		The universal property of \Cref{kahler_algebraifold} indicates why this definition is natural, and \Cref{algebraifold_vs_LR} explains why we prefer it over a setting like the one used by Beggs and Majid, 
	\item Some of our examples of algebraifolds, most notably separable finitely generated field extensions and fibred differential geometry via submersions (\Cref{field_ext,parametric}), have apparently not been considered before,
	\item \Cref{dual_bases_algebraifold} on dual bases for vector fields and $1$-forms gives a useful computational tool which to some extent acts as a substitute for charts.
	\item Our consideration of differentials of algebraifold maps in \Cref{cats}, which crucially informs our new algebraic definition of formal lines and geodesics in \Cref{geodesics}, seems to be novel.
\end{itemize}

\section{Algebraifolds}
\label{sec:algebraifolds}

In order to help us find out what a good setting for an algebraic generalization of differential geometry will be, let us indicate which aspects of manifold-based differential geometry we would like to consider in the first place.
These are the following:

\begin{enumerate}[label=(\alph*)]
	\item\label{vector_item} Vector fields, $1$-forms, and tensors,
	\item Lie derivatives,
	\item Affine connections on vector bundles,
	\item Metrics, the Levi--Civita connection and the usual curvature tensors,
	\item Geodesics.
\end{enumerate}

Of course, many other structures exist and are important in suitable flavours of differential geometry and physics, such as principal bundles, jet bundles or spinor bundles.
We will not consider those in this paper.
The above stuff display the following relevant structures and properties, which we will make crucial use of in the algebraic generalization:

\begin{enumerate}[resume,label=(\alph*)]
	\item\label{duality_item} Duality between vector fields and $1$-forms.
		In particular, this means e.g.~that a $(1,1)$-tensor can be defined as a linear map from vector fields to vector fields, or likewise from $1$-forms to $1$-forms, or as a bilinear map from vector fields and $1$-forms to smooth functions. 
	\item Every smooth function $f$ has a $1$-form derivative $df$.
	\item Vector fields form a Lie algebra with the Lie derivative as the Lie bracket.
	\item\label{LC_item} The Levi--Civita connection is the only torsion-free connection compatible with the metric.
	\item\label{pullback_item} $1$-forms can be pulled back along smooth maps.
	\item Vector bundles and connections on them can be pulled back along smooth maps.
	\item\label{geodesic_item} A geodesic is a smooth curve whose tangent vector is covariant constant along the curve, in the sense that pulling back the connection as well as the tangent vectors along itself results in a covariant constant section of a vector bundle on $\R$.
\end{enumerate}

We now begin the technical development by axiomatizing the duality between vector fields and $1$-forms in algebraic terms.
Postulating this duality will turn out to be sufficient for developing algebraic generalizations of all the other things listed above.

\subsection*{Derivations}

For a manifold $M$, it is well-known that the vector fields on $M$ are in canonical bijection with the set of $\R$-linear derivations on $C^\infty(M)$~\cite[Theorem~2.72]{lee}.
Indeed every vector field $v$ defines a derivation obtained by sending every function $f$ to its directional derivative along $v$, and the surprising fact is that every derivation is of this form for a unique $v$, postulating only $\R$-linearity and the Leibniz rule, but not requiring continuity of any kind.

It is thus natural to identify vector fields with derivations, and to use this as the definition of an algebraic generalization of vector fields as follows~\cite{sardanashvily}.

\begin{defn}
	Let $k$ be a commutative ring and $\ZLA$ a commutative $k$-algebra.
	Then a \newterm{derivation} on $\ZLA$ is a $k$-linear map $D: \ZLA \to \ZLA$ such that the \newterm{Leibniz rule}
	\begin{equation}
		\label{leibniz_rule}
		D(a_1 a_2) = D(a_1) a_2 + a_1 D(a_2) 
	\end{equation}
	holds for all $a_1, a_2 \in \ZLA$.
	We write $\D_{\ZLA}$ for this set of derivations, leaving the base ring $k$ implicit.
\end{defn}

The set $\D_{\ZLA}$ is easily seen to be an $\ZLA$-module with respect to the usual pointwise addition and scalar multiplication.
In particular, the scalar multiplication by $a \in \ZLA$ is given by
\[
	(a D)(a') \coloneqq a D(a').
\]

More generally, we can also consider derivations $D : \ZLA \to \Mod$ taking values in any $\ZLA$-module $\Mod$ by defining them as $k$-linear maps satisfying the same Leibniz rule~\eqref{leibniz_rule}, but now considered as an equation in $\Mod$.
These derivations form an $\ZLA$-module in the same way, and we denote it by $\Der_k(\ZLA,\Mod)$.
In particular,
\[
	\D_{\ZLA} = \Der_k(\ZLA, \ZLA).
\]

\subsection{Duality with $1$-forms}

Considering $\D_{\ZLA}$ as the algebraic analogue version of vector fields, we now would like to define the $\ZLA$-module of $1$-forms as its dual module,
\begin{equation}
	\label{1forms}
	\Omega_{\ZLA} := \D_{\ZLA}^* = \Modules{\ZLA}(\D_{\ZLA}, \ZLA),
\end{equation}
in order to generalize the familiar duality between vector fields and $1$-forms on a manifold.
However, the duality in the manifold case is quite a bit stronger than this: since the vector fields on $M$ are sections of a vector bundle, \Cref{SSS} shows that they form a \newterm{finitely generated projective (fgp)} module over $C^\infty(M)$, and this implies that the duality between vector fields and $1$-forms takes the rather strong form explicated by \Cref{fgp_dual}.
This strong form of the duality is precisely what allows us to consider a $(1,1)$-tensor in the various equivalent forms mentioned at~\ref{duality_item}.

So in order to generalize this duality to the algebraic setting, we simply impose it as an additional requirement on the algebra $\ZLA$, and in fact as essentially\footnote{For the Levi--Civita connection, we will also impose $2 \in k$ to be invertible (\cref{two_inv}), but this seems rather innocuous since all of our examples of interest will satisfy $\Q \subseteq k$.} the \emph{only} requirement that we will need in order to generalize all of~\ref{vector_item}--\ref{LC_item}.\footnote{The remaining items~\ref{pullback_item}--\ref{geodesic_item} then refer to maps, which we will consider later in \Cref{cats}.}

\begin{defn}
	\label{algebraifold}
	For a commutative ring $k$, a \newterm{$k$-algebraifold} is a commutative $k$-algebra $\ZLA$ such that the $\ZLA$-module $\D_{\ZLA}$ is fgp.
\end{defn}

\begin{ex}
	\label{mfd_ex}
	Of course, our motivating example of an algebraifold is the $\R$-algebra $C^\infty(M)$ for any smooth manifold $M$.
	The derivations correspond to the vector fields on $M$, which form an fgp module over $C^\infty(M)$ as part of \Cref{SSS}.\footnote{\label{mfd_dim_fn}
		Although~\Cref{SSS} assumes finitely many connected components for $M$, the direction of implication that is relevant here is easily proven generally, provided that one defines ``manifold'' such that its components have uniformly bounded sets of generating vector fields, which is necessary in order for the module of vector fields is not finitely generated.

		Although we do not feel the need to be fully precise about the definition of ``manifold''---given that manifolds are not the main protagonists of this paper---it is worth pointing out that certain definitions in the theory of manifolds need to be handled with great care in relation to the non-connected case.
	For example, already when setting up the very definition of manifold itself, one runs into the dilemma of whether the dimension of a manifold should be the same across different components.
		Different standard textbook accounts of the theory disagree on this. For example, Lang's~\cite[Chapter~II]{lang} defines manifolds while allowing varying target spaces for charts, and therefore varying dimension across components, which has the advantage of equipping the category of manifolds with coproducts.
		The more common choice, as made e.g.~by de Rham~\cite[Chapter~I]{derham} and Lee~\cite[Chapter~1]{lee}, is to fix the dimension from the start, in which case all components of a manifold must have the same dimension.}
\end{ex}

We will consider other examples of algebraifolds in \Cref{algebraifold_examples}.
For now, let us argue next that this definition is quite natural also from an algebraic geometry point of view.

\subsection{A universal property for $1$-forms}

Upon seeing \Cref{algebraifold}, an algebraic geometer might object that it looks inelegant: there is an established algebraic notion of differential $1$-form, which is usually regarded as the ``correct'' definition due to the universal property that it enjoys, and this is a priori different from our more mundane-looking~\eqref{1forms}.
Let us first recall this standard definition before addressing this objection.

\begin{defn}
	\label{kahler_forms}
	Let $\ZLA$ be a commutative algebra over a commutative ring $k$.
	Then the $\ZLA$-module of \newterm{Kähler differentials} $\widehat{\Omega}_{\ZLA}$ is the $\ZLA$-module with generators $\{da \: : \: a \in \ZLA\}$ subject to the relations
	\[
		\begin{aligned}
			d(a + b)	&= da + db	\\
			d(ra)		&= r \, da	\\
			d(ab)		&= (da) b + a \, db
		\end{aligned}
	\]
	for all $a, b \in \ZLA$ and $r \in k$.
\end{defn}

The universal property of $\widehat{\Omega}_{\ZLA}$ is that it is the initial $\ZLA$-module equipped with a derivation $d : \ZLA \to \widehat{\Omega}_{\ZLA}$.
In other words, $\widehat{\Omega}_{\ZLA}$ classifies derivations with values in any other module $\Mod$, in the sense that there is a natural bijection
\[
	\begin{tikzcd}[row sep=4pt,column sep=6pt]
		\Der_\R(\ZLA, \Mod) \; \ar[draw=none,"\displaystyle{\cong}" description]{r}	& \; \Modules{\ZLA}(\widehat{\Omega}_{\ZLA}, \Mod)	\\
		\qquad h \circ d & h. \qquad \ar[mapsto]{l}.
	\end{tikzcd}
\]
This holds essentially by definition of $\widehat{\Omega}_{\ZLA}$, where the map from right to left is given by composing a $\ZLA$-module homomorphism by the universal derivation $d$.

In contrast to this, our definition~\eqref{1forms} does not enjoy any obvious universal property.
Indeed for any manifold $M$ with at least one component of positive dimension, the $C^\infty(M)$-module of Kähler differentials $\widehat{\Omega}_{C^\infty(M)}$ over $k = \R$ does \emph{not} coincide with the module of differential $1$-forms, essentially\footnote{We refer to~\cite[Theorem~16]{gomez} for a formal proof; see also~\cite{osborn} for general results on relations between Kähler differentials.} since the equation
\[
	d(e^x) = e^x \, dx	
\]
holds as an equation in $\Omega_{C^\infty(M)}$ but not in $\widehat{\Omega}_{C^\infty(\R)}$, which is intuitive since the relations between Kähler differentials that have been imposed in \Cref{kahler_forms} let us evaluate $d$ on polynomial expressions but not on infinite series.
Therefore the standard Kähler differentials are \emph{not} an algebraic generalization of differential $1$-forms in the sense that we would like to have, although this is achieved by our mundane~\eqref{1forms}.

However, it turns out that there is a surprisingly simple way to get the best of both worlds.
This consists of modifying the universal property of Kähler differentials in such a way that it \emph{does} apply to the module of $1$-forms in the manifold case, and more generally to our $\Omega_{\ZLA}$ for algebraifolds.
In the manifold case, this is the following known result.

\begin{thm}[{\cite[Theorem~11.43]{nestruev}}]
	\label{manifold_univ_der}
	Let $M$ be a manifold and $\Omega_{C^\infty(M)}$ the $C^\infty(M)$-module of differential $1$-forms on $M$.
	Then the map
	\begin{align}
	\begin{split}
		\label{univ_dev_mfd}
		C^\infty(M)	& \longrightarrow \Omega_{C^\infty(M)}	\\
		f		& \longmapsto df
	\end{split}
	\end{align}
	is an $\R$-derivation, and it is the universal $\R$-derivation with values in fgp modules: for every fgp $C^\infty(M)$-module $\Mod$, we have a bijection
	\[
		\begin{tikzcd}[row sep=4pt,column sep=6pt]
			\Der_\R(C^\infty(M), \Mod) \; \ar[draw=none,"\displaystyle{\cong}" description]{r}	& \; \Modules{C^\infty(M)}(\Omega_{C^\infty(M)}, \Mod)	\\
			\qquad h \circ d & h. \qquad \ar[mapsto]{l}
		\end{tikzcd}
	\]
\end{thm}

In fact, the cited~\cite[Theorem~11.43]{nestruev} is a bit stronger than this: it proves the universal property for a larger class of $C^\infty(M)$-modules which Nestruev calls \emph{geometric}.
We will not need this stronger form here.
To see how our formulation follows from his, it is enough to note that every fgp module satisfies this condition (essentially by \Cref{SSS}).

For algebraifolds in general, we have the following result.

\begin{thm}
	\label{kahler_algebraifold}
	For a commutative ring $k$ and a commutative $k$-algebra $\ZLA$, the following are equivalent:
	\begin{enumerate}
		\item\label{algebraifold_item}
			$\ZLA$ is a $k$-algebraifold.
		\item\label{fgp_kahler}
			There exists an fgp module $\Omega_{\ZLA}$ together with a $k$-derivation
			\[
				d \: : \: \ZLA \longrightarrow \Omega_{\ZLA}
			\]
			which is universal for derivations on fgp modules: for every fgp $\ZLA$-module $\Mod$, we have a bijection
			\beq
				\label{algebraifold_omega}
				\begin{tikzcd}[row sep=4pt,column sep=6pt]
					\Der_k(\ZLA, \Mod) \; \ar[draw=none,"\displaystyle{\cong}" description]{r}	& \; \Modules{\ZLA}(\Omega_{\ZLA}, \Mod)	\\
					\qquad h \circ d & h. \qquad \ar[mapsto]{l}
				\end{tikzcd}
			\eeq
	\end{enumerate}
	If these conditions hold, then we can take $\Omega_\ZLA = \D_{\ZLA}^*$ with universal derivation given by
	\begin{align}
	\begin{split}
		\label{univ_der_eval}
		\ZLA & \longrightarrow \D_{\ZLA}^*	\\
		a & \longmapsto \left( v \mapsto v(a) \right).
	\end{split}
	\end{align}
\end{thm}

As usual for a universal property, condition~\ref{fgp_kahler} characterizes $\Omega_{\ZLA}$ up to unique isomorphism.

\begin{proof}
	Assume first that~\ref{fgp_kahler} holds.
	Then considering $\Mod = \ZLA$ in~\eqref{algebraifold_omega} gives an isomorphism
	\[
		\D_{\ZLA} \cong \Modules{\ZLA}(\Omega_{\ZLA}, \ZLA) = \Omega_{\ZLA}^*.
	\]
	Thus $\D_{\ZLA}$ is the dual of an fgp module and therefore itself fgp, which already proves~\ref{algebraifold_item}.

	Conversely, assuming~\ref{algebraifold_item} we will construct an isomorphism
	\beq
		\label{derM_hom}
		\Der_k(\ZLA, \Mod) \cong \Modules{\ZLA}(\D_{\ZLA}^*, \Mod)
	\eeq
	that is natural in fgp modules $\Mod$, and the fact that it is of the claimed form~\eqref{algebraifold_omega} with $\Omega_\ZLA = \D_\ZLA^*$ is then automatic by abstract nonsense~\cite[Proposition~2.4.8]{riehl}.
	For $\Mod = \ZLA$, there clearly is such an isomorphism since the canonical map
	\beq
		\label{canonical_doubledual}
		\D_{\ZLA} \longrightarrow \D_{\ZLA}^{**}
	\eeq
	is an isomorphism by the fgp assumption.
	This canonical isomorphism is obviously natural with respect to module endomorphisms of $\ZLA$, and it follows that we obtain a natural bijection~\eqref{derM_hom} for any finitely generated free module $\Mod = \ZLA^{\oplus n}$.
	Finally, every finitely generated projective module $\Mod$ is a direct summand of some $\ZLA^{\oplus n}$, and this implies the claim in general.

	It remains to characterize the universal derivation as~\eqref{univ_der_eval}.
	By abstract nonsense, the universal derivation is the counterpart of $\id_{\D_{\ZLA}^*}$ on the left-hand side of~\eqref{derM_hom}.
	In order to show that this is given by~\eqref{univ_der_eval}, consider the naturality diagram
	\[
		\begin{tikzcd}
			\Der_k(\ZLA, \D_{\ZLA}^*)	\ar[draw=none,"\displaystyle{\cong}" description]{r} \ar[swap]{d}{v^{**} \circ -}	& \Modules{\ZLA}(\D_{\ZLA}^*, \D_{\ZLA}^*) \ar{d}{v^{**} \circ -}	\\
			\D_{\ZLA} = \Der_k(\ZLA, \ZLA) \ar[draw=none,"\displaystyle{\cong}" description]{r}	& \Modules{\ZLA}(\D_{\ZLA}^*, \ZLA)
		\end{tikzcd}
	\]
	where $v^{**} \in \D_{\ZLA}^{**}$ is the element associated to any $v \in \D_{\ZLA}$, namely the map $\D_{\ZLA}^* \to \ZLA$ given by evaluation on $v$.
	Our goal is to show that the candidate universal derivation~\eqref{univ_der_eval} in the top left corresponds to the identity in the top right.
	Since the module homomorphisms $v^{**}$ separate the elements of $\D_{\ZLA}^*$, it is enough to show that this correspondence holds at the bottom after composing both by $v^{**}$ as indicated.
	Since $v^{**}$ is given by evaluation on $v$, this results in the map $a \mapsto v(a)$ on the left-hand side, which is $v$ itself.
	On the right-hand side, we trivially get $v^{**}$.
	This indeed reproduces the canonical map~\eqref{canonical_doubledual}, which is how we had constructed the bottom isomorphism.
\end{proof}

\begin{ex}
	In the manifold case, the fact that the universal derivation is given by~\eqref{univ_der_eval} amounts to the standard fact that differential $1$-forms can be seen as dual to vectors, where for $f \in C^\infty(M)$ its differential $df$ acts on a vector field $v$ by the directional derivative, $(df)(v) = v(f)$, which is exactly the action of $v$ as a derivation on $f$.
\end{ex}

\begin{rem}
	\label{dual_bases_algebraifold}
	Sometimes it is convenient to choose dual bases $u_1, \ldots, u_n$ of $\D_{\ZLA}$ and $\xi_1, \ldots, \xi_n$ of $\Omega_{\ZLA}$ in the sense of~\eqref{duality}.
	More specifically, the proof of \Cref{fgp_dual} shows that $(\xi_i)$ can actually be \emph{any} generating set of $\Omega_{\ZLA}$ as an $\ZLA$-module.
	Since the differentials $da$ for $a \in \ZLA$ generate $\Omega_{\ZLA}$, we can take $\xi_i = d a_i$ for suitable $a_1, \ldots, a_n \in \ZLA$.
	Like this, due to $(d a_i)(v) = v(a_i)$ the dual basis properties~\eqref{duality} and~\eqref{duality2} read
	\begin{equation}
		\label{dual_bases_algebraifold_eq}
		\sum_{i=1}^n v(a_i) \, u_i = v, \qquad \sum_{i=1}^n \eta(u_i) \, d a_i = \eta
	\end{equation}
	for all $v \in \D_{\ZLA}$ and $\eta \in \Omega_{\ZLA}$.
	In geometrical terms, this amounts to a way of writing every vector field as a canonical $\ZLA$-linear combination of generating vector fields.
	As we will see in the discussion of differentials of algebraifold maps in \Cref{lem_explicit_pullback} and after, a dual basis in this form can play a role similar to that of charts in the manifold context.
\end{rem}

The following example suggests that we may think of the $a_i$ in \Cref{dual_bases_algebraifold} as generalized ``coordinates'' and of the $u_i$ as generalized ``coordinate vector fields''.
In fact, this way of thinking is quite indicative of the typical uses of~\eqref{dual_bases_algebraifold_eq} for algebraifolds, as appearing e.g.~in \Cref{lem_explicit_pullback}.

\begin{ex}
	\label{Rn_dual_bases}
	For the manifold $M = \R^n$, we can take the coordinate functions and their associated vector fields,
	\begin{equation}
		\label{Rn_dual_bases_eq}
		a_i = x_i, \qquad \qquad u_i = \frac{\partial}{\partial x_i},
	\end{equation}
	and the dual basis relation~\eqref{dual_bases_algebraifold_eq} amounts to the usual way of writing every vector field and $1$-form as a linear combination of coordinate vector fields and $1$-forms,
	\[
		v = \sum_{i=1}^n v(x_i) \, \frac{\partial}{\partial x_i}, \qquad
		\eta = \sum_{i=1}^n \eta\left(\frac{\partial}{\partial x_i}\right) \, dx_i.
	\]
\end{ex}

What is special about~\eqref{Rn_dual_bases_eq} is that in this case, the matrix $(u_i(a_j))_{i,j=1}^n$ with entries from $\ZLA$ is the identity matrix.
That this does not hold in general is the main difference between coordinates and coordinate vector fields on the one hand and the dual bases approach on the other.
For the latter, this matrix seems to be an important invariant of the dual bases under consideration.
Although its appearance is an additional complication relative to coordinates and coordinate vector fields, a nice feature of this approach is that it applies globally.

\begin{rem}
	\label{algebraifold_vs_LR}
	One may wonder whether it is really necessary to consider \emph{all} derivations as the analogue of vector fields or whether one could consider an fgp $\ZLA$-module $\E$ of vector fields as a separate piece of structure which acts on $\ZLA$ by derivations through an $\ZLA$-linear Lie algebra homomorphism $\E \to \D_{\ZLA}$.
	This would result in a definition similar to a (commutative) Lie--Rinehart algebra~\cite{huebschmann3}.
	Modulo some technical variations, this is essentially the route taken with the \emph{differential triads} of Mallios~\cite{malliosI,malliosII}, the \emph{differential calculi} of Beggs and Majid~\cite{BM} and the \emph{Rinehart spaces} of Pessers and van der Veken~\cite{PV}, all apparently independently.

	Indeed this kind of setting seems meaningful and may be of interest for geometry and physics.
	We have opted for our definition, which fixes $\D_{\ZLA} = \Der_k(\ZLA, \ZLA)$ as the algebraic analogue of the space of vector fields, for several reasons:
	\begin{itemize}
		\item We want to have a formalism that exactly specializes to manifold differential geometry when instantiated on the $\R$-algebras $C^\infty(M)$.
			Lie--Rinehart algebras generalize Lie algebroids instead.
			Thus using such a formalism does \emph{not} give an algebraic analogue of vector fields, but rather of something a bit different.
		\item Proving the existence of dual bases of the form~\eqref{dual_bases_algebraifold_eq}, which will come in handy in \Cref{cats} in the consideration of differentials of algebraifold maps, requires $\Omega_\ZLA$ to be generated by $1$-forms of the form $da$.
			While this automatically holds for algebraifolds, it is clearly not automatic for the module $\E^*$ in a Lie--Rinehart algebra.\footnote{Pessers and van der Veken consider this an additional condition on their \emph{Rinehart spaces} called \emph{regularity}~\cite[Section~3]{PV}.}
		\item In case that $\E \to \D_{\ZLA}$ is injective, or without loss of generality if $\E \subseteq \D_{\ZLA}$, then it may still be the case that $\E = \Der_{\hat{k}}(\ZLA,\ZLA)$, where
			\[
				\hat{k} \coloneqq \{ a \in \ZLA \mid D(a) = 0 \;\; \forall D \in \E \},
			\]
			is the ring of constants.
			Whenever this happens, $\ZLA$ can be treated as an algebraifold over $\hat{k}$.

			For example in the manifold case, a general submodule $\E \subseteq \D_{C^\infty(M))}$ closed under the Lie bracket corresponds to a foliation of $M$.
			The ring of constants $\hat{k} \subseteq C^\infty(M)$ then consists of the functions constant along the leaves.
			We have $\E = \Der_{\hat{k}}(C^\infty(M), C^\infty(M))$ as soon as the leaves are the fibres of a submersion, in which case we are back in the situation of \Cref{parametric}.
			This brings us back to the first item above: algebraifolds are the algebraic analogue of manifolds, while the more general Lie--Rinehart algebras can also be used as algebraic analogues of foliations.
	\end{itemize}
\end{rem}

\section{Examples of algebraifolds and standard form}
\label{algebraifold_examples}

\begin{ex}
	For a non-compact connected manifold $M$, one may be tempted to consider a number of variations on the algebraifold $C^\infty(M)$ from \Cref{mfd_ex} by imposing additional conditions on the smooth functions under consideration.
	However, neither of our attempts at constructing an $\R$-algebraifold in this way has succeeded.
	This includes the following:
	\begin{enumerate}
		\item\label{compact_support}
			$C^\infty_{c,1}(M)$, the algebra of smooth functions which are constant outside of a compact set, is \emph{not} an algebraifold.\footnote{This algebra is the unital version of the non-unital algebra of compactly supported functions $C^\infty_c(M)$ in that $C^\infty_{c,1}(M) = C^\infty_c(M) \oplus \R$. We consider this version, as our definition of algebraifold requires unitality to begin with (although this could potentially be relaxed).}

			To see this, note first that the module of derivations $\D_{C^\infty_{c,1}(M)}$ again consists of \emph{all} vector fields on $M$.
			Indeed it is clear that every vector field defines such a derivation. Conversely, that every $\R$-linear derivation is of this form follows by the same reasoning as in the $C^\infty(M)$ case~\cite[Chapter~9]{nestruev}.
			Furthermore, this module is not finitely generated: if it was, then the module of vector fields modulo the submodule of compactly supported vector fields would also be finitely generated, but this is not the case as it is infinite-dimensional as an $\R$-vector space, and acting by an element of $C^\infty_{c,1}(M)$ on this module only amounts to multiplication by a scalar.
		\item $C^\infty_b(M)$, the algebra of bounded smooth functions on $M$, is \emph{not} an algebraifold either.

			To see this, we argue first that the module of derivations $\D_{C^\infty_b(M)}$ consists of the compactly supported vector fields on $M$. 
			Indeed every such vector field clearly defines a derivation on $C^\infty_b(M)$.
			Conversely, using the same arguments as in~\cite[Chapter~9]{nestruev} again shows that every derivation is given by differentiation along a vector field.
			Now if the support of such a vector field is not compact, then we can choose a divergent sequence of points contained in it, and find a bounded smooth function $f \in C^\infty_b(M)$ whose directional derivatives along the vector field are not bounded by making it oscillate more sufficiently quickly towards infinity.

			Having established that $\D_{C^\infty_b(M)}$ corresponds to the compactly supported vector fields on $M$, it is again enough to note that this module is not finitely generated.
			This is simply because any submodule generated by finitely many compactly supported vector fields can only contain vector fields whose support is contained in the union of the supports of the generators, which is again a compact set.
	\end{enumerate}
	Our next counterexample is conjectural.
	\begin{enumerate}[resume]
		\item The unital \emph{Schwartz space} $\mathcal{S}(\R^n)$ consists of all smooth functions $\R^n \to \R$ whose partial derivatives of order $\ge 1$ all decay faster than polynomial.\footnote{Similar to~\ref{compact_support}, this is the unital version of the standard Schwartz space.}
			We suspect that this is not an algebraifold either (for $n \ge 1$).

			In more detail, it seems plausible that the module of derivations $\D_{\mathcal{S}(\R^n)}$ consists of the smooth vector fields on $\R^n$ that are bounded in (Euclidean) length by a polynomial.
			Indeed in one direction, it is easy to see that every such polynomial defines a derivation on $\mathcal{S}(\R^n)$.
			Conversely, by the same arguments as before every derivation corresponds to a smooth vector field $v$; what still needs to be shown is that such a vector field is necessarily upper bounded by a polynomial.
			Assuming that this is the case, the same argument as in \ref{compact_support}, quotienting the module of vector fields upper bounded by a polynomial by the submodule of vector fields with components in $\mathcal{S}(\R^n)$.
	\end{enumerate}
	These negative examples suggest that the embedding of the category of manifolds into the category of $\R$-algebraifolds (\Cref{cats}) is quite rigid.
\end{ex}

So on the positive side, what are other examples of algebraifolds?

\begin{ex}
	\label{poly_ring}
	For any $n \in \N$, the polynomial ring $\ZLA \coloneqq k[x_1,\ldots,x_n]$ over any commutative ring $k$ is a $k$-algebraifold.
	Indeed a $k$-derivation $D : \ZLA \to \ZLA$ is uniquely determined by its values on the variables $x_1, \dots, x_n$ by $k$-linearity and the Leibniz rule, and these values can be chosen completely arbitrarily.
	Therefore $\D_{\ZLA}$ is free of rank $n$, and in particular fgp.
	A convenient basis is given by the partial derivative operators $\frac{\partial}{\partial x_i}$ for $i = 1, \ldots, n$.
\end{ex}

\begin{ex}
	\label{algebraic_varieties}
	The \newterm{Zariski-Lipman conjecture}~\cite{hochster}\footnote{We thank MathOverflow user \emph{jg1896} for pointing this out to us.} is an open problem in commutative algebra, which in our language asks when a finitely generated algebra without zero divisors is an algebraifold.

	To state it in detail, let $k$ be a field of characteristic zero.
	Then if $\ZLA$ is a finitely generated \emph{regular}\footnote{Intuitively, regularity is a condition on an algebra that has the same geometrical flavour as the ``locally looking like $\R^n$'' property of a manifold.} $k$-algebra without zero divisors, it is known that $\ZLA$ is an algebraifold since already the module of Kähler differentials $\widehat{\Omega}_{\ZLA}$ is fgp~\cite[{}15.2.11]{MR_nnr}, which is sufficient to conclude that $\ZLA$ is an algebraifold with
	\begin{equation}
		\label{iskahler}
		\Omega_{\ZLA} = \widehat{\Omega}_{\ZLA}
	\end{equation}
	by \Cref{kahler_algebraifold}.
	The long-standing open problem is now whether the converse holds as well~\cite{hochster}: if $\ZLA$ is a finitely generated $k$-algebraifold without zero divisors, does this imply regularity?

	This is known to be true under additional conditions~\cite{hochster,BG,tipler}.
	The fact that no counterexamples are known suggests that our definition of algebraifold is at least quite close to regularity in this class of algebras.
	This makes intuitive sense as both properties encode a similar type of smoothness.
\end{ex}

\begin{ex}
	\label{field_ext}
	Let $k$ be a field and $k \subseteq \ZLA$ a separable field extension, for example an extension of characteristic zero.
	Then this makes the larger field $\ZLA$ into a $k$-algebraifold if and only if the extension is finitely generated.
	For example, the field of rational functions in $d$ variables $k(x_1, \ldots, x_n)$ is finitely generated, and it is therefore a $k$-algebraifold.

	To see why the relevant condition is finite generation, note first that the module of Kähler differentials $\widehat{\Omega}_{\ZLA}$ is a vector space over $\ZLA$ with a basis given by a transcendence basis of the extension~\cite[Proposition~6.1.15]{liu}.
	Since
	\[
		\D_{\ZLA} = \left( \widehat{\Omega}_{\ZLA} \right)^*,
	\]
	the claim follows by the fact that the dual of an $\ZLA$-vector space is finite-dimensional if and only if the original $\ZLA$-vector space is.
	Note that we again have~\eqref{iskahler} in this case.

	For an example that is less trivial than $k(x_1, \dots, x_n)$, consider an \emph{elliptic function field} like
	\[
		\ZLA \coloneqq k(x, \sqrt{x^3 + 1}),
	\]
	where $k$ is an arbitrary field of characteristic different from $2$ and $3$.
	This is the field whose elements are formal expressions in $x$ and $\sqrt{x^3 + 1}$ with denominators allowed, subject to the usual rules for how to calculate with such expressions.\footnote{Fully formally, $\ZLA$ is defined as the field of fractions of the quotient ring $k[x, y] / (y^2 - x^3 - 1)$, where the variable $y$ plays the role of $\sqrt{x^3 + 1}$. Since the polynomial $y^2 - x^3 - 1$ is irreducible, this quotient ring indeed has no zero divisors, so that its field of fractions can be formed.}
	The $\ZLA$-vector space of derivations $\D_{\ZLA}$ is one-dimensional:
	with $\{x\}$ as a transcendence basis, the associated single basis vector of $\D_{\ZLA}$ is the derivative operator $\frac{\partial}{\partial x}$, which acts on the generating elements via
	\[
		\frac{\partial}{\partial x} (x) = 1, \qquad \quad \frac{\partial}{\partial x} \left( \sqrt{x^3 + 1} \right) = \frac{3x^2}{2\sqrt{x^3 + 1}},
	\]
	and extends uniquely from there by $k$-linearity and the Leibniz rule.

	Equivalently, we could also consider
	\[
		\ZLA = k(y, \sqrt[3]{y^2 - 1}),
	\]
	which is an isomorphic elliptic function field, where the isomorphism is given by $x \mapsto \sqrt[3]{y^2 - 1}$ in one direction and by $y \mapsto \sqrt{x^3 + 1}$ in the other.
	This isomorphism is a fully algebraic analogue of a coordinate transformation.
	Using now $\{y\}$ as a transcendence basis, we can also use $\frac{\partial}{\partial y}$ as a basis vector for $\D_{\ZLA}$, with a similarly explicit action on the generating elements.
	The change of basis between the two choices of basis vector can be expressed as
	\[
		\frac{\partial}{\partial y} = \frac{2 y}{3 x^2} \, \frac{\partial}{\partial x}.
	\]
\end{ex}

\Cref{poly_ring,algebraic_varieties,field_ext} are quite special in that they are finitely generated as algebras over $k$ and have no zero divisors, which is related to the fact that our algebraifold $1$-forms coincide with the Kähler differentials~\eqref{iskahler}.
In addition to those of the form $C^\infty(M)$, there also exist other interesting algebraifolds of a more analytical flavour in which these additional properties do not hold.

\begin{ex}
	Let $S \subseteq \C^n$ be any open connected subset.
	Then $\mathcal{O}(S)$, the $\C$-algebra of holomorphic functions on $S$, is an algebraifold.
	Indeed every derivation $D : \mathcal{O}(S) \to \mathcal{O}(S)$ is of the form $D = \sum_{i=1}^n f_i \frac{\partial}{\partial z_i}$ for coefficient functions $f_1, \dots, f_n \in \mathcal{O}(S)$~\cite{BP}, and therefore the module of derivations is free of rank $n$.
	In fact, this even works with $k = \Z$ as the base ring instead, due to the surprising result that every (merely $\Z$-linear) derivation $\mathcal{O}(S) \to \mathcal{O}(S)$ is $\C$-linear~\cite[Corollary~5.2]{BZ}.
\end{ex}

\begin{ex}
	Let $C^\omega(\R^n)$ be the $\R$-algebra of real-analytic functions on $\R^n$.
	Then this is an algebraifold, since every $\R$-linear derivation on $C^\omega(\R^n)$ is given by a real-analytic vector field~\cite[Theorem~5.1]{grabowski}, and the module of these vector fields is free of rank $n$ (as follows by writing them in components).

	More generally, if $M$ is a real-analytic manifold, then $\D_{C^\omega(M)}$ is again given by the real-analytic vector fields on $M$~\cite[Theorem~5.1]{grabowski}.
	It therefore seems plausible that $C^\omega(M)$ is an $\R$-algebraifold as well, but we have not yet been able to prove this.
\end{ex}

So far, all of our examples of $\R$-algebraifolds involve functions ``finer'' than smooth, in the sense that $\ZLA$ is contained in the algebra of smooth functions in each case.
What about algebraifolds which properly extend algebras of smooth functions?
Our next example indicates that finding such algebraifolds is not an easy task.

\begin{ex}
	For a non-interesting example of an algebraifold which properly contains $C^\infty(\R^n)$, consider the $\R$-algebra of all continuous functions $C(\R^n)$.	
	This algebra admits no nonzero derivations at all.\footnote{See e.g.~\href{https://ncatlab.org/nlab/show/derivation\#DerOfContFuncts}{https://ncatlab.org/nlab/show/derivation\#DerOfContFuncts}} 
	Thus this algebra is technically an $\R$-algebraifold, but clearly of an uninteresting kind.
\end{ex}

To exclude pathological examples like these, we frequently impose the following additional condition, a variant of the \emph{connectedness condition} of Beggs and Majid~\cite[Definition.1.1(4)]{BM}.

\begin{defn}
	A $k$-algebraifold is in \newterm{standard form} if:
	\begin{enumerate}
		\item The homomorphism $k \to \ZLA$ which defines scalar multiplication is injective.
		\item For all $a \in \ZLA$,
			\[
				D(a) = 0 \quad \forall D \in \D_{\ZLA} \qquad \Longrightarrow \qquad a \in k.
			\]
	\end{enumerate}
\end{defn}

The idea is that the elements of $k$ should be exactly the constants, geometrically interpreted as those functions which have vanishing derivative in all directions.
Being in standard form is essentially a without loss of generality assumption for the following reason.
The set
\beq
	\label{hatk}
	\hat{k} \coloneqq \{ a \in \ZLA \mid D(a) = 0 \;\; \forall D \in \D_{\ZLA} \}
\eeq
is a subring of $\ZLA$ that we call the \newterm{ring of constants}.
Every $k$-derivation $D$ is automatically a $\hat{k}$-derivation by definition of $\hat{k}$ and a straightforward application of the Leibniz rule.
In other words, we have
\[
	 \D_{\ZLA} = \Der_{\hat{k}}(\ZLA, \ZLA),
\]
and this shows that $\ZLA$ is also a $\hat{k}$-algebraifold, now in standard form by construction.
One can thus assume standard form without loss of generality.

\begin{ex}
	For the $\R$-algebra $C(\R^n)$, which has no nonzero derivations, the ring of constants is $C(\R^n)$ itself.
	Therefore the standard form is trivial in this case, in the sense that the algebraifold coincides with the base ring, and this makes precise the idea that it is an uninteresting example.
\end{ex}

\begin{ex}
	\label{mfd_sf}
	For $M$ a manifold, the $\R$-algebraifold $C^\infty(M)$ is in standard form if and only if $M$ is connected.
	Indeed, the ring of constants~\eqref{hatk} contains precisely those functions that are constant on connected components.
	In particular, if $M$ has $n \in \N$ components, then its ring of constants is isomorphic to $\R^n$ with componentwise multiplication.
\end{ex}

\begin{ex}
	\label{poly_ring2}
	For any commutative ring $k$ of characteristic zero, the polynomial ring $\ZLA = k[x_1, \ldots, x_n]$ as considered in \Cref{poly_ring} is clearly an algebraifold in standard form.
	The situation is different in characteristic $\ell > 0$, since then the ring of constants is the subalgebra generated by the $\ell$-th powers of the variables $x_i^\ell$.
\end{ex}

\begin{ex}
	\label{field_ext2}
	Continuing on from \Cref{field_ext}, let $k \subseteq \ZLA$ be a finitely generated separable field extension in characteristic zero.
	Then we have that $\hat{k}$ is
	\[
		\hat{k} = \{ a \in \ZLA \mid a \text{ is algebraic over } k \},
	\]
	which is a subfield of $\ZLA$ known as the \emph{field of constants}~\cite{suzuki}.
	In particular, $\ZLA$ is in standard form if and only if the only elements of $\ZLA$ algebraic over $k$ are those of $k$ itself.
	For example, this is automatically the case if $k$ is algebraically closed.

	In characteristic $\ell > 0$, the situation is more complicated as in \Cref{poly_ring2}, since then for example in $\ZLA \coloneqq k(x)$ we have
	\[
		\frac{\partial (x^\ell)}{\partial x} = 0,
	\]
	so that $x^\ell$ also belongs to the ring of constants.
\end{ex}

We now return to the question of how to find interesting algebraifolds over $\R$ which property extend algebras of smooth functions.

\begin{ex}
	\label{colombeau}
	For a smooth manifold $M$, we write $\mathscr{G}(M)$ for the \newterm{Colombeau algebra} of generalized functions on $M$.\footnote{There are various versions of the Colombeau algebra, but we limit our discussion to the ``special'' one as it seems to be simplest to define and best understood. We refer to~\cite{GKOS} for a nice textbook account of the main variants, namely the ``special'' and ``full'' Colombeau algebras.}
	These were defined for open sets $M \subseteq \R^n$ by Colombeau~\cite{colombeau1,colombeau2} and generalized to manifolds in~\cite{AB,dRD}.
	Subsequently, Kunzinger and Steinbauer developed foundational aspects of distributional differential geometry in terms of these algebras~\cite{KS}.
	See also~\cite{GKOS} for a comprehensive textbook account.

	Since the Colombeau algebra satisfies a sheaf condition~\cite[Proposition~3.2.3]{GKOS}, to define $\mathscr{G}(M)$ it is enough to do so under the additional assumption that $M \subseteq \R^n$ is an open subset.
	In this case, $\mathscr{G}(M)$ is usually defined as the quotient of the algebra
	\begin{equation}
		\label{colombeau_before_quotient}
		\left\{ (f_\eps)_{\eps \in (0,1]} \in C^\infty(M)^{(0,1]} \:\bigg|\: \forall K, \alpha \: \exists r > 0 : \: \sup_{x \in K} \, \left| (\partial^\alpha f)(x) \right| = O(\eps^{-r}) \right\},
	\end{equation}
	where $K \subseteq M$ ranges over all compact subsets and $\alpha$ over all multiindices over $1, \dots, n$, by the ideal of \emph{negligible} elements,
	\begin{equation}
		\left\{ (f_\eps)_{\eps \in (0,1]} \in C^\infty(M)^{(0,1]} \:\bigg|\: \forall K, \alpha \: \forall s > 0 : \: \sup_{x \in K} \, \left| (\partial^\alpha f)(x) \right| = O(\eps^s) \right\}.
	\end{equation}
	So roughly speaking, $\mathscr{G}(M)$ contains families of functions where the quotienting ensures that only the behaviour as $\eps \to 0$ is of interest, and where these families may diverge in the limit $\eps \to 0$.
	Clearly $C^\infty(M)$ can be embedded into $\mathscr{G}(M)$ by mapping every smooth function to the corresponding constant sequence.
	But also the space of Schwartz distributions on $M$ embeds (non-canonically) into $\mathscr{G}(M)$~\cite[Theorem~3.2.10]{GKOS}.
	To first approximation, one may think of $\mathscr{G}(M)$ as an algebra which lifts Schwartz distributions to a setting in which their multiplication is well-defined and satisfies some intuitively desirable properties.

	It is known that that the module of derivations of $\mathscr{G}(M)$ over $k = \R$ coincides with the module of similarly defined \emph{generalized vector fields}, which are a special case of generalized sections of vector bundles over $M$~\cite[Theorem~7]{KS}.
	These modules are known to be fgp~\cite[Theorem~3.2.22]{GKOS}, and therefore $\mathscr{G}(M)$ is an algebraifold over $\R$.
	It is not in standard form for any manifold $M \neq \emptyset$, since the ring of constants consists of all those elements that can be represented by families $(f_\eps)_{\eps \in [0,1)}$, where the $f_\eps$ are constant on the connected components of $M$~\cite[Section~1.2.4]{GKOS}.

	Due to the technical complexity of working with generalized function algebras like the Colombeau algebra, we will not discuss them any further in this paper.
	So let us merely note that the framework for differential geometry with generalized functions of~\cite{KS} can be viewed as an instance of the algebraifold formalism.\footnote{This does not mean that analytical considerations become obsolete in any sense, but only that familiarity with algebraifolds can help with understanding the purely formal aspects of this flavour of differential geometry.}
	We hope that this observation can contribute to its broader adoption, which seems of interest especially in physics thanks to the possibility of considering distributional metrics and tensors.
	In addition, it would be interesting to see whether other variants of the Colombeau algebras\footnote{In particular, the full Colombeau algebra on a smooth manifold and the special Colombeau algebra with smooth parameter dependence~\cite{BK} come to mind.} are algebraifolds as well, in which case our formalism would immediately yield associated flavours of differential geometry.
\end{ex}

Our next two conjectural examples extend the algebra of smooth functions in a way that feels similar to the function fields from \Cref{field_ext}.

\begin{conj}
	For a smooth manifold $M$, let $C^\infty_d(M)$ be the $\R$-algebra of densely defined smooth functions $f : (D_f \subseteq M) \to \R$ with domain $D_f \subseteq M$ an embedded submanifold, and where two such functions are identified if they coincide on the intersection of their domains.
	Then it seems plausible that $C^\infty_d(M)$ is an $\R$-algebraifold, with $\D_{C^\infty_d(M)}$ consisting of all densely defined smooth vector fields on $M$.
\end{conj}

\begin{conj}
	For a smooth manifold $M$, let
	\[
		S \coloneqq \{ f \in C^\infty(M) \mid f^{-1}(0) \text{ has empty interior} \}
	\]
	be the set of ``almost invertible'' smooth functions.
	Then $S$ does not contain any zero divisors, since $fg = 0$ for $f \in S$ implies that $g$ must vanish on a dense subset, so that $g = 0$ by continuity.
	Furthermore, $S$ clearly contains $1$ and is closed under multiplication.
	Therefore the localization
	\[
		\ZLA \coloneqq S^{-1} C^\infty(M) 
	\]
	makes sense and is an $\R$-algebra containing $C^\infty(M)$.
	We suspect that this $\R$-algebra is an algebraifold as well, with module of derivations $\D_\R(S^{-1} C^\infty(M)) = S^{-1} \D_\R(C^\infty(M))$.
\end{conj}

As a final class of examples, we note that the algebraifold formalism can natively treat fibred manifolds.

\begin{ex}
	\label{parametric}
	Let $f : M \to N$ be a smooth map between smooth manifolds.
	Then composing with $f$ turns smooth functions on $N$ into smooth functions on $M$, and this is a ring homomorphism $C^\infty(f) : C^\infty(N) \to C^\infty(M)$.
	This homomorphism makes $C^\infty(M)$ into a $C^\infty(N)$-algebra, and we claim that it is an algebraifold provided that $f$ is a submersion.

	To see this, we need to show that the module of $C^\infty(N)$-linear derivations
	\[
		\D_{C^\infty(N)}(C^\infty(M))
	\]
	is fgp.
	To this end, recall first that the $\R$-linear derivations on $C^\infty(M)$ correspond to the vector fields on $M$.
	Then $\D_{C^\infty(N)}(C^\infty(M))$ is the submodule consisting of the \emph{vertical} vector fields, since taking the directional derivative along a vector field $v$ is a $C^\infty(N)$-linear operation if and only if all fibrewise constant smooth functions have vanishing derivative along $v$.
	Since the vertical vector fields are the sections of a vector bundle over $M$, we conclude that they form an fgp module over $C^\infty(M)$ by~\Cref{SSS}.
	Therefore we are indeed dealing with an algebraifold.

	For this algebraifold to be in standard form, a necessary (and possibly sufficient) condition is for the set
	\[
		\{ y \in N \mid \text{the fibre $f^{-1}(y)$ is connected} \}
	\]
	to be dense in $N$.
	In this case, any smooth function whose directional derivative along all vertical vector fields vanishes must be constant on connected fibres\footnote{To see this, note that connectedness for a fibre $f^{-1}(y)$ implies path-connectedness, and consider a local trivialization on each point of a path to obtain that such a function must be constant along the path.}, from which the claim follows by density.
	For example, if the submersion $f$ is a smooth fibre bundle with connected fibre, then this algebraifold is in standard form.

	The gist of this example is that the algebraifold formalism automatically covers the idea of \emph{fibred} differential geometry, in light of the surjective submersion $f : M \to N$ making $M$ into a fibred manifold over $N$.
	We expect similar constructions to be possible for our other examples as well, and possibly for algebraifolds in general.
\end{ex}

For the rest of the paper, we will focus on the finitely generated field extensions (\Cref{field_ext}) and fibred manifolds (\Cref{parametric}) as our main examples of algebraifolds in addition to those of the form $C^\infty(M)$.
To end this section, we consider a construction of algebraifolds which generalizes the disjoint union of smooth manifolds.
This also formalizes the most trivial part of the Colombeau algebra construction, namely the formation of the power $C^\infty(M)^{(0,1]}$ in~\eqref{colombeau_before_quotient}, in general.

\begin{prop}
	\label{coprods}
	Let $(\ZLA_i)_{i \in I}$ be a family of algebraifolds in standard form over commutative rings $(k_i)_{i \in I}$, and such that the modules of derivations $\Der_{k_i}(\ZLA_i)$ are uniformly finitely generated.
	Then the direct product $\prod_{i \in I} \ZLA_i$ is an algebraifold in standard form over the direct product $\prod_{i \in I} k_i$.
\end{prop}

\begin{proof}
	Let us write $\ZLA \coloneqq \prod_{i \in I} \ZLA_i$ and $k \coloneqq \prod_{i \in I} k_i$ for brevity, and consider each $\ZLA_i$ as included in $\ZLA$ as the $i$-th component.
	We then first consider the map
	\begin{align}
	\begin{split}
		\label{coprod_der}
		\D_k(\ZLA) & \longrightarrow \prod_{i \in I} \D_{k_i}(\ZLA_i) \\
		D & \longmapsto \left( a \mapsto D(a)_i \right)_{i \in I}.
	\end{split}
	\end{align}
	Indeed $a \mapsto D(a)_i$ is a derivation on $\ZLA_i$ as a straightforward consequence of the Leibniz rule and $D(1_i) = 0$, which holds because $1_i$ is an idempotent in $\ZLA$.\footnote{If $e = e^2$ is any idempotent in a ring and $D$ any derivation, then applying the Leibniz rule to $D(e) = D(e^2) = D(e^3)$ shows $D(e) = 0$. Hence $e$ belongs to the ring of constants, meaning that $D(e a) = e D(a)$ for all $a$.}

	We claim that~\eqref{coprod_der} is an isomorphism of $\ZLA$-modules with respect to the obvious componentwise $\ZLA$-module structure on the right-hand side.
	Indeed the $\ZLA$-linearity of the map is obvious.
	To see that~\eqref{coprod_der} is an injection, suppose that $D$ is such that $D(a)_i = 0$ for all $i$ and all $a \in \ZLA_i$.
	Since $1_i$ belongs to the ring of constants, it follows that even $D(a) = 0$.
	Therefore $D$ vanishes on the elements lying in any $\ZLA_i$, which is enough to conclude that $D = 0$ by the fact that each $1_i$ belongs to the ring of constants.
	Finally the surjectivity is easy to see: for any family of derivations $(D_i)_{i \in I}$ with $D_i \in \D_{k_i}(\ZLA_i)$, we can define a derivation on $\ZLA$ by
	\[
		D((a_i)_{i \in I}) \coloneqq \left( D_i(a_i) \right)_{i \in I},
	\]
	and it is clear that this recovers the $D_i$ under~\eqref{coprod_der}.
	The isomorphism~\eqref{coprod_der} is therefore established, and it also follows that $k$ is the ring of constants for $\ZLA$ by the standard form assumption on the $\ZLA_i$.

	With the isomorphism~\eqref{coprod_der} at hand, it follows that $\D_k(\ZLA)$ is an fgp $\ZLA$-module by the uniform finite generation hypothesis.
	The latter guarantees that the $n \in \N$ in \Cref{fgp_dual}\ref{fgp_retract} can be chosen independent of $i$, and hence we obtain the same property of $\D_k(\ZLA)$ by choosing the module homomorphisms $u$ and $\xi$ componentwise.
\end{proof}

\begin{ex}
	Let $(M_i)_{i \in I}$ be a family of $n$-dimensional\footnote{If one does define ``manifold'' such that different components are allowed to have different dimensions (\Cref{mfd_dim_fn}), }
	connected smooth manifolds.
	Then their disjoint union $\coprod_{i \in I} M_i$ is an $n$-dimensional smooth manifold as well, and we have a canonical isomorphism
	\[
		C^\infty \left( \coprod_{i \in I} M_i \right) \cong \prod_{i \in I} C^\infty(M_i).
	\]
	Now provided that the $M_i$ have uniformly finite sets of generating vector fields, \Cref{coprods} applies and shows that $C^\infty \left( \coprod_{i \in I} M_i \right)$ is an algebraifold in standard form over $\R^I$, in accordance with \Cref{mfd_ex,mfd_sf}.

	This class of examples also shows that the statement of \Cref{coprods} is not true without the hypothesis of uniform finite generation.
	Indeed for the direct product $\ZLA \coloneqq \prod_{n \in \N} C^\infty(\R^n)$, the isomorphism~\eqref{coprod_der} still holds, and therefore a derivation on $\ZLA$ consists of a sequence of vector fields in all dimensions.
	Clearly the $\ZLA$-module consisting of such sequences is not finitely generated.
\end{ex}

\section{Tensors, connections and curvature}
\label{sec:diff_geom}

We now continue to develop basic differential geometry, in the sense of pseudo-Riemannian geometry, with algebraifolds.
As outlined at the beginning of \Cref{sec:algebraifolds}, this includes defining the notions of tensor, connection and curvature, including metric tensors, the Levi--Civita connection and the Riemann and Ricci curvature tensors; geodesics will be considered in \Cref{geodesics}.
Since our developments here are largely completely parallel to standard textbook material in the manifold case, we will keep this somewhat brief.
The main purpose of this section is to show that these notions \emph{can} be developed just the same for any algebraifold, with the additional requirement that $2 \in k$ should be invertible in order for the Levi--Civita connection to be defined (\Cref{two_inv}).

\begin{nota}
	Throughout this section, we fix an algebraifold $\ZLA$ over a commutative ring $k$ and use the shorthand notations
	\[
		\D \coloneqq \D_{\ZLA}, \qquad \D^* \coloneqq \Omega_{\ZLA}, \qquad \otimes_\ZLA \coloneqq \otimes.
	\]
\end{nota}

We write $\D^*$ instead of $\Omega_{\ZLA}$ for the module of $1$-forms in order to emphasize the duality with $\D$, which is more natural in the context of tensor calculus than the universal property of \Cref{kahler_algebraifold}.

\subsection{Tensors}

When applied to the $\R$-algebraifolds of the form $C^\infty(M)$, the following specializes to one of the standard definitions of tensor field in the manifold case~\cite[Section~7.3]{lee}; see also~\cite[Corollary~7.5.7]{conlon} for an explicit statement of the equivalence with the ``pointwise'' definition of tensor on a smooth manifold.

\begin{defn}
	\label{defn_tensor}
	On a $k$-algebraifold $\ZLA$, and for $r,s \in \N$, a \newterm{tensor of rank $(r,s)$} is an $\ZLA$-multilinear map
	\beq
		\label{einstein_algebra_tensor}
		\underbrace{\D^* \times \cdots \times \D^*}_{r \textrm{ factors}} \times \underbrace{\D \times \cdots \times \D}_{s \textrm{ factors}} \longrightarrow \ZLA.
	\eeq
	We write $\T^r_s$ for the set of these tensors.
\end{defn}

The usual universal property of tensor products let us equivalently consider such a map as an ordinary $\ZLA$-linear map
\[
		\underbrace{\D^* \otimes \cdots \otimes \D^*}_{r \textrm{ factors}} \otimes \underbrace{\D \otimes \cdots \otimes \D}_{s \textrm{ factors}} \longrightarrow \ZLA.
\]
In the manifold case, this is closely related to the fact that the sections functor from vector bundles to fgp modules preserves tensor products (\Cref{SSS}).

By the duality theory of fgp modules, as recalled in \Cref{fgp} and \eqref{fgp_hom} in particular, we can always turn an occurrence of $\D^*$ on the left into an occurrence of $\D$ on the right and vice versa.
Hence we could equivalently define a tensor of rank $(r,s)$ as an $\ZLA$-linear map
\beq
	\label{s_to_r_tensor}
	\underbrace{\D \otimes \cdots \otimes \D}_{s \textrm{ factors}} \longrightarrow \underbrace{\D \otimes \cdots \otimes \D}_{r \textrm{ factors}},
\eeq
or yet again equivalently as and $\ZLA$-linear map
\beq
	\label{r_to_s_tensor}
	\underbrace{\D^* \otimes \cdots \otimes \D^*}_{r \textrm{ factors}} \longrightarrow \underbrace{\D^* \otimes \cdots \otimes \D^*}_{s \textrm{ factors}},
\eeq
or finally simply as an element of
\[
	\underbrace{\D \otimes \cdots \otimes \D}_{r \textrm{ factors}} \otimes \underbrace{\D^* \otimes \cdots \otimes \D^*}_{s \textrm{ factors}}.
\]
This shows that in terms of dual bases as in~\eqref{duality} and \Cref{dual_bases_algebraifold}, every tensor can be written as an $\ZLA$-linear combination of elementary tensors formed out of these basis elements.
Of course, the various descriptions listed above can also be mixed, in the sense that only some of the factors of $\D$ and $\D^*$ are moved to the other side.
This will turn out to be useful below in~\eqref{contraction}.

\begin{ex}
	The \newterm{Kronecker tensor} $\delta \in \T^1_1$ is defined in terms of~\eqref{einstein_algebra_tensor} as
	\[
		\delta(\xi, v) \coloneqq \xi(v),
	\]
	or simply as the identity map $\D \to \D$ in terms of~\eqref{s_to_r_tensor} or~\eqref{r_to_s_tensor}.
	The fact that these define the same tensor is a consequence of \Cref{fgp_dual}, and is particularly easy to see in terms of the graphical calculus of \Cref{fgp_rem}\ref{string_diagrams}, where the Kronecker tensor is denoted by a vertical line labelled by $\D$.
\end{ex}

\begin{ex}
	Clearly $(0,1)$-tensors are just $1$-forms, $\T^0_1 = \D^*$.
	In particular, for every $a \in \ZLA$ we recover the derivative $1$-form $da \in \D^*$ as the $\ZLA$-liner map $\D \to \ZLA$ defined by
	\[
		(da)(v) \coloneqq v(a) \qquad \forall v \in \D.
	\]
\end{ex}

\begin{ex}
	\label{tensors_examples}
	\begin{enumerate}
		\item Of course, for an algebraifold of the form $C^\infty(M)$, the tensors are exactly the tensors on the manifold $M$ in the usual sense.
		\item\label{poly_tensors}
			For a polynomial ring $\ZLA = k[x_1, \ldots, x_n]$, the tensors of rank $(r,s)$ are the elements of
			\[
				\underbrace{(\ZLA^n)^* \otimes \cdots \otimes (\ZLA^n)^*}_{r \textrm{ factors}} \otimes \underbrace{\ZLA^n \otimes \cdots \otimes \ZLA^n}_{s \textrm{ factors}},
			\]
			which is canonically isomorphic to $\ZLA^{n^{r+s}}$.
			In other words, due to the fact that we have designated coordinate directions, a tensor of rank $(r,s)$ is just an $(r+s)$-dimensional array of polynomials in $n$ variables.
		\item For a finitely generated field extension $k \subseteq \ZLA$ of transcendence degree $n = \dim_{\ZLA}(\D)$, the tensors of rank $(r,s)$ form an $\ZLA$-vector space given by
			\[
				\underbrace{\D \otimes \cdots \otimes \D}_{r \textrm{ factors}} \otimes \underbrace{\D^* \otimes \cdots \otimes \D^*}_{s \textrm{ factors}},
			\]
			which is an $\ZLA$-vector space of dimension $n^{r+s}$.
		\item\label{parametric_tensors}
			For a smooth submersion $f \colon M \to N$ of smooth manifolds as in \Cref{parametric}, we saw that $\D$ is the module of vertical vector fields on $M$.
			Therefore a tensor of rank $(r,s)$ is an element of the $r$-fold tensor power of the vector bundle of vertical vector fields on $M$ tensored by the $s$-fold power of its dual.
			Intuitively, this description amounts to saying that a tensor is a ``smoothly varying family of tensors'' on the fibres of $f$.

			For a more explicit description, note first that the module of $1$-forms $\D^*$ corresponds to the module of $1$-forms on $M$ modulo those which vanish on vertical vectors.\footnote{To see this, it is enough to show that every $C^\infty(M)$-linear map from vertical vector fields back to $C^\infty(M)$ can be extended to such a map on all vector fields.
			Thanks to the existence of partitions of unity this can be done locally, so that $f$ can be assumed to be a coordinate projection $\R^{\ell+n} \to \R^\ell$ without loss of generality. But in this case the statement is clear, as the module of vertical vector fields is a direct summand of the module of all vector fields. \label{extension_argument}}
			By the same token, a tensor of rank $(r,s)$ is then an equivalence class of ``vertical'' tensors of rank $(r,s)$ on $M$, where vertical means that it must be possible to write the tensor as a finite sum of elementary tensors with vertical vectors in the contravariant slots, and the equivalence amounts to quotienting by those vertical tensors which vanish on vertical vectors in the covariant slots.
	\end{enumerate}
\end{ex}

\subsection{Algebraic structure of tensors}

Returning now to the general theory, it is clear that the sum of two tensors of rank $(r,s)$ is again such a tensor, defined in terms of pointwise addition.
Furthermore, if $a \in \ZLA$ and $T \in \T^r_s$, then we also have a tensor $a T \in \T^r_s$ defined by
\[
	(a T)(\xi_1, \ldots, \xi_r, v_1, \ldots, v_s) \coloneqq a \, T(\xi_1, \ldots, \xi_r, v_1, \ldots, v_s,).
\]
In this way, $\T^r_s$ becomes an $\ZLA$-module.
Of course, this is also obvious from the isomorphism
\[
	\T^r_s \cong \underbrace{\D \otimes \cdots \otimes \D}_{r \textrm{ factors}} \otimes \underbrace{\D^* \otimes \cdots \otimes \D^*}_{s \textrm{ factors}}
\]
already noted above.
The definition of tensor as a multilinear map also results in a straightforward construction of the tensor products of tensors,
\[
	S \in \T^r_s, \qquad T \in \T^{r'}_{s'} \qquad \Longrightarrow \qquad S \otimes T \in \T^{r + r'}_{s + s'},
\]
simply amounting to the usual multiplication of multilinear maps which adds up the arities.
It makes the collection of all tensors into a bigraded algebra with $\ZLA$ itself in degree $(0,0)$.

An important operation on tensors in differential geometry is contraction.
Indeed we can consider a given tensor $T \in \T^{r+1}_{s+1}$ as a module homomorphism of the form
\begin{equation}
	\label{contraction}
	\underbrace{\D^* \otimes \cdots \otimes \D^*}_{r \textrm{ factors}} \otimes \underbrace{\D \otimes \cdots \otimes \D}_{s \textrm{ factors}} \longrightarrow \D^* \otimes \D,
\end{equation}
and composing with the canonical evaluation homomorphism $\varepsilon : \D^* \otimes \D \to \ZLA$, which is yet another incarnation of the Kronecker tensor, defines an element of $\T^r_s$ referred to as the \newterm{contraction} of $T$.
As usual, this can be applied with respect to any contravariant and any covariant tensor slot.
For tensors of rank $(1,1)$, this contraction results in an element of $\ZLA$ known as the \newterm{Hattori-Stallings trace}~\cite{hattori}.
For example, every $(1,1)$-tensor of the form $\xi \otimes v$ for $\xi \in \D^*$ and $v \in \D$ contracts to $\xi(v)$.
In the graphical calculus of \Cref{fgp_rem}\ref{string_diagrams}, contraction can be depicted particularly easily as ``looping'' an output wire back to an input wire.
Applying this to the Kronecker tensor gives a particularly interesting result.

\begin{defn}
	\label{dimension}
	The \newterm{dimension} $\dim(\ZLA)$ of a $k$-algebraifold $\ZLA$ is the contraction (Hattori-Stallings trace) of the Kronecker tensor $\delta$.
\end{defn}

Given the terminology, one might expect the dimension to be a number.
This is indeed the case in the following sense.

\begin{lem}
	If $\ZLA$ is in standard form, then $\dim(A) \in k$.
\end{lem}

So in general, the dimension is an element of the ring of constants.

\begin{proof}
	Although this uses the properties of Lie derivatives introduced below, we present the proof here as these properties may already be familiar to the reader from the manifold case.

	Since Lie derivatives commute with contraction, it is enough to show that the Kronecker tensor has vanishing Lie derivative along any $u \in \D$.
	Using the other properties of Lie derivatives from \Cref{lie_derivative} as well as~\eqref{L_on_xi} and~\eqref{L_on_tensor}, we get for any $\xi \in \D^*$ and $v \in \D$,
	\[
		\L_u(\delta)(\xi, v) = \L_u(\xi(v)) - \L_u(\xi)(v) - \xi(\L_u(v)) = 0,
	\]
	as was to be shown.
\end{proof}

For explicit computations of the dimension, it can be convenient to choose dual bases of $\D$ and $\D^*$ as in~\ref{dual_bases_algebraifold}, in terms of which the dimension can be explicitly computed as
\beq
	\label{dim_trace}
	\dim(\ZLA) = \sum_{i=1}^n da_i(u_i) = \sum_{i=1}^n u_i(a_i),
\eeq
which is exactly the trace of the matrix considered after \Cref{Rn_dual_bases}.

\begin{ex}
	Let us see what the dimension is in our running examples.
	\begin{enumerate}
		\item For $C^\infty(M)$ for a manifold $M$, the dimension is the scalar function which maps every component of $M$ to its manifold dimension. This is a standard property of the Kronecker tensor.
		\item Similarly, the dimension of the polynomial $k$-algebraifold $k[x_1, \ldots, x_n]$ is $n$, as follows from~\eqref{dim_trace} upon choosing $v_i \coloneqq \frac{\partial}{\partial x_i}$ and $\xi_i \coloneqq dx_i$.
		\item If $k \subseteq \ZLA$ is a finitely generated field extension of characteristic zero, then the dimension is the transcendence degree, and in particular again an integer.
			Indeed as in \Cref{field_ext,field_ext2}, $\D$ is a free $\ZLA$-vector space of rank equal to the transcendence degree, and this is what~\eqref{dim_trace} computes.

			If $k$ has characteristic $\ell > 0$ and the extension is separable, then the dimension is still given by the transcendence degree for the same reason, but now considered as an element of $\Z / \ell$.
		\item Continuing on from \cref{parametric}, let $f : M \to N$ be a smooth submersion between smooth manifolds, and consider $C^\infty(M)$ as a $C^\infty(N)$-algebraifold.
			Then the characterization of derivations given in \cref{parametric} implies that the dimension $\dim(C^\infty(M))$ is the function which maps each point $p \in M$ to the dimension of a sufficiently small neighbourhood in the fibre containing that point, or equivalently to the dimension of the component of $M$ containing $p$ minus the dimension of the component of $N$ containing $f(p)$.
	\end{enumerate}
\end{ex}

\subsection{Lie derivatives of tensors}

We now describe how every derivation $v \in \D$ acts canonically on every tensor $T \in \T^r_s$.
In the manifold case, this action is known as \newterm{Lie derivative}.
First, for $u,v \in \D$, we obtain another derivation given by the Lie bracket,
\[
	[u,v](a) \coloneqq u(v(a)) - v(u(a))	\qquad \forall a \in \ZLA.
\]
Some straightforward calculation shows that this is indeed a derivation again, and that this Lie bracket operation turns $\D$ into a Lie algebra over $k$.
Moreover, it satisfies the equation
\beq
	\label{LR}
	[u, av] = a [u,v] + u(a) v.
\eeq
On general tensors, we have the following construction of Lie derivatives.

\begin{lem}
	\label{lie_derivative}
	Every derivation $u \in \D$ extends uniquely to a family of $k$-linear maps
	\[
		\L_u \: : \: \T^r_s \longrightarrow \T^r_s	 \qquad \forall r,s \in \N
	\]
	such that the following hold:
	\begin{enumerate}
		\item The Leibniz rule: for arbitrary tensors $S \in \T^r_s$ and $T \in \T^{r'}_{s'}$,
			\[
				\L_u(S \otimes T) = \L_u(S) \otimes T + S \otimes \L_u(T).
			\]
		\item On $a \in \ZLA$, we have
			\[
				\L_u(a) = u(a).
			\]
		\item\label{L_on_w} On $v \in \D$, we have
			\[
				\L_u(v) = [u,v].
			\]
		\item $\L_u$ commutes with tensor contraction.
	\end{enumerate}
\end{lem}

\begin{proof}
	Let us consider $(0,1)$-tensors $\xi \in \D^*$ first. On these, we must have, for every $w \in \D$,
	\[
		u(\xi(v)) = \L_u(\xi(v)) = \L_u(\xi)(v) + \xi(\L_u(v))
	\]
	where the second equation holds by the Leibniz rule and the commutation with contractions.
	We can use \ref{L_on_w} in order to thus compute
	\beq
		\label{L_on_xi}
		\L_u(\xi)(v) = u(\xi(v)) - \xi([u,v]).
	\eeq
	Using~\eqref{LR}, it is easy to see that this expression is $\ZLA$-linear in $v$, and we therefore obtain an element $\L_u(\xi) \in \D^*$.

	If $T \in \T^r_s$ is now an arbitrary tensor, then a similar calculation shows that
	\begin{align}
		\label{L_on_tensor}
		\L_u(T)(\xi_1, \ldots, \xi_r, v_1, \ldots, v_s) ={} & u\left( T(\xi_1, \ldots, \xi_r, v_1, \ldots, v_s) \right) \nonumber\\
								& - \sum_{i=1}^r T(\xi_1, \ldots, \L_u(\xi_i), \ldots, \xi_r, v_1, \ldots, v_s) \\
								& - \sum_{j=1}^s T(\xi_1, \ldots, \xi_r, v_1, \ldots, \L_u(v_j), \ldots, v_s). \nonumber
	\end{align}
	It is a straightforward calculation to see that this defines a tensor again, and that this is the only definition which makes all required properties hold.
\end{proof}

\subsection{Connections}

The notion of connection from differential geometry has a standard generalization to connections on modules over rings~\cite[Definition~8.2.2]{MS}, which we now simply instantiate for algebraifolds.
We continue using the shorthand notation $\D \coloneqq \D_{\ZLA}$, assuming that $\ZLA$ is an algebraifold over a commutative ring $k$.

\begin{defn}
	Let $\Mod$ be an $\ZLA$-module. Then a \newterm{connection} on $\ZLA$ is a map
	\begin{equation}
		\label{connection}
		\nabla : \D \times \Mod \longrightarrow \Mod, \qquad (u,x) \longmapsto \nabla_u x
	\end{equation}
	which is $k$-bilinear and has the following additional properties:
	\begin{enumerate}
		\item $\nabla$ is $\ZLA$-linear in its first argument: for all $a \in \ZLA$ and $u \in \D$ and $x \in \Mod$,
			\[
				\nabla_{au} x = a \nabla_u x.
			\]
		\item $\nabla$ satisfies the Leibniz rule in its second argument: for all $a \in \ZLA$ and $u \in \D$ and $x \in \Mod$,
			\[
				\nabla_u(a x) = u(a) \hspace{1pt} x + a \nabla_u x.
			\]
	\end{enumerate}
\end{defn}

\begin{rem}
	\label{connection_form}
	By the duality between $\D$ and $\D^*$ and the $\ZLA$-linearity of a connection in the first argument, we could also consider a connection as a map $\nabla : \Mod \to \D^* \otimes \Mod$ satisfying the equation $\nabla(a x) = a \nabla(x) + da \otimes x$, and where one may think of $\D^* \otimes \Mod$ as the module of $\Mod$-valued $1$-forms.
	This picture if often used in algebraic treatments of connections, while we slightly prefer~\eqref{connection} due to its proximity to the usual differential geometric notation.
\end{rem}

\begin{ex}
	For $\ZLA = C^\infty(M)$ with a manifold $M$ and $\Mod$ the module of sections of a vector bundle, this is literally the usual notion of (linear) connection.
\end{ex}

For the remainder of this section, we limit ourselves to the case $\Mod = \D$, meaning that we consider the algebraic generalization of connections on the tangent bundle.
The following standard observation is then also as in the manifold case, with essentially identical proof.

\begin{lem}
	\label{connections_difference}
	Any two connections differ by a tensor of rank $(1,2)$.
\end{lem}

\begin{proof}
	Let $\nabla$ and $\nabla'$ be two connections. Then it needs to be shown that their difference, which is a $k$-bilinear map $\D \times \D \to \D$ given by
	\[
		(u,v) \longmapsto \nabla_u v - \nabla'_u v,
	\]
	is $\ZLA$-linear in both arguments. This is clear for the first argument by the $\ZLA$-linearity assumption on connections. In the second argument, we get that for $a \in \ZLA$,
	\begin{align*}
		(\nabla_u - \nabla'_u)(av)	& = u(a) \hspace{1pt} v + a \, \nabla_u v - u(a) \hspace{1pt} v - a \, \nabla'_u v	\\
						& = a \hspace{1pt} (\nabla_u - \nabla'_u) \hspace{1pt} v,
	\end{align*}
	as was to be shown.
\end{proof}

\begin{ex}
	Here's what connections amount to in some of our examples.
	\begin{enumerate}
		\item In the manifold case $C^\infty(M)$, connections are of course connections in the usual sense (on the tangent bundle of $M$), and \Cref{connections_difference} amounts to expressing one connection in terms of another via a tensor. When one of these is the (locally defined) standard connection with respect to a chart, then the components of this tensor are the Christoffel symbols.
		\item \label{christoffel} Similarly, for the polynomial ring $\ZLA = k[x_1,\ldots,x_n]$, the partial derivative operators $\partial_i = \frac{\partial}{\partial x_i}$ form a basis of $\D$ over $\ZLA$, so that a generic element of $\D$ has the form
			\[
				u = \sum_i u_i \partial_i.
			\]
			In this notation, $\ZLA$ has a standard connection determined uniquely by the formula
			\beq
				\label{standard_connection}
				\nabla_{\partial_i} v := \sum_j \frac{\partial v_j}{\partial x_i} \, \partial_j,
			\eeq
			which simply amounts to taking the componentwise partial derivative of $v$ in the $i$-direction.
			Of course, for $k = \R$ this is just the usual standard connection on $\R^n$, restricted to polynomial vector fields.

			Using~\Cref{connections_difference}, one can now express every other connection on $k[x_1,\ldots,x_n]$ in the form $\nabla + \Gamma$, where $\Gamma$ is a tensor of rank $(1,2)$. Its components are once again the Christoffel symbols, which now must be polynomials rather than arbitrary smooth functions.
		\item For a finitely generated separable field extension $k \subseteq \ZLA$, let us write $\partial_i = \frac{\partial}{\partial x_i}$ for the basis of $\D$ associated to a transcendence basis $\{x_1, \ldots, x_n\} \subseteq \ZLA$.
			Then again the very same formula~\eqref{standard_connection} defines a canonical connection on $\ZLA$.
			Naturally this connection depends on the choice of transcendence basis, as one can see by explicit calculation in the elliptic function field of \Cref{field_ext}.
		\item If $f : M \to N$ is again a smooth submersion as in~\Cref{parametric}, then a connection on the associated algebraifold is an equivalence class of connections on the vector bundle of vertical vector fields on $M$, where the equivalence is $\nabla \sim \nabla'$ if and only if $\nabla_u = \nabla'_u$ for all vertical vector fields $u$.
			This can be seen by the same argument as in \Cref{extension_argument}.
	\end{enumerate}
\end{ex}

In general, a connection extends to a covariant derivative operator on tensors of any rank in the exact same way as in the manifold case. For example if $T$ is a tensor of rank $(1,1)$, then for any $u,v \in \D$ and $\xi \in \D^*$, we obtain
\[
	(\nabla_u T)(v,\xi) = u (T(v, \xi)) - T(\nabla_u v, \xi) - T(v, \nabla_u \xi),
\]
and a straightforward calculation again proves $\ZLA$-linearity in all three of $u$, $v$ and $\xi$, so that $\nabla_u T$ is indeed a tensor, now of rank $(1,2)$.

\subsection{Torsion and curvature}

The torsion of a connection $\nabla$ is the $(0,2)$-tensor defined by
\[
	T(u,v) := \nabla_u v - \nabla_v u - [u,v],
\]
and one can verify the $\ZLA$-linearity by a straightforward calculation similar to the proof of \Cref{connections_difference}.
The curvature instead is the $(1,3)$-tensor
\beq
	\label{curvature_algebraic}
	R \: : \: \D \times \D \times \D \longrightarrow \D
\eeq
given by
\[
	R(u,v) w = \nabla_u \nabla_{v} w - \nabla_{v} \nabla_u w - \nabla_{[u,v]} w,
\]
where one writes $R(u,v') w$ instead of $R(u,v,w)$ to underline the intuition that for every $u$ and $v$, the resulting $(1,1)$-tensor $R(u,v)$ is an endomorphism $\D \to \D$, which in the manifold case encodes the holonomy transformation for parallel transport along an infinitesimal rectangle spanned by $u$ and $v$.
Again a straightforward calculation verifies the relevant multilinearity for $R$ to be a tensor.

\subsection{Metric tensors}

Again we have the same definition as in the manifold case.

\begin{defn}
	A \newterm{metric tensor} or simply \newterm{metric} on an algebraifold $\ZLA$ is a symmetric $(0,2)$-tensor
	\[
		g : \D \otimes \D \longrightarrow \ZLA,
	\]
	which is nondegenerate: there must be a $(2,0)$-tensor
	\[
		g^{-1} : \D^* \otimes \D^* \longrightarrow \ZLA
	\]
	that is inverse to $g$, in the sense that the contraction of the $(2,2)$-tensor $g \otimes g^{-1}$ is the identity $(1,1)$-tensor $\D \to \D$.
\end{defn}

\begin{rem}
	\begin{enumerate}
		\item By the assumed symmetry of $g$, it does not matter which pair of slots of $g \otimes g^{-1}$ is contracted.
		\item An equivalent formulation is that a metric tensor is an $\ZLA$-module isomorphism $\D \cong \D^*$, which can be constructed by mapping $u \in \D$ to the contraction of $u \otimes g$.
			This is the \newterm{musical isomorphism}.
		\item Note that our definition does not impose any positivity or signature conditions on the metric.
			This is the same as in other algebraic approaches like Geroch's~\cite{geroch}.
			We explore some initial ideas on how to improve on this situation in \Cref{signature}.
		\item If $2 \in k$ is not invertible, for example if $k$ has characteristic $2$, then it may be more appropriate to define a metric as a \emph{quadratic} form $\D \to \ZLA$~\cite[\S{}1.4]{pfister}.
			We will not dwell on this issue and focus instead on the case where $2$ is invertible (\Cref{two_inv}).
	\end{enumerate}
\end{rem}

\begin{ex}
	\label{metric_concrete}
	\begin{enumerate}
		\item With $\ZLA = C^\infty(M)$, metrics are of course smooth pseudo\hyp{}Riemannian metrics on $M$ in the standard sense.
		\item For a polynomial ring $\ZLA = k[x_1,\ldots,x_n]$, making use of~\Cref{poly_tensors} lets us identify the metrics with those symmetric matrices with polynomial entries whose inverses are also polynomial matrices, resulting in the form
			\[
				g = \sum_{i,j=1}^n g_{ij} \hspace{1pt} dx_i \otimes dx_j,
			\]
			known from coordinate calculations on manifolds. Thus, for example with $n = 2$, and writing $k[x,y]$ for simplicity, we have such a metric given by
	\beq
		\label{poly_metric}
		g =	\left(\begin{matrix}
				1 & x \\
				x & 1 + x^2 
			\end{matrix}\right),
			\qquad\quad
		g^{-1} =	\left(\begin{matrix}
					1 + x^2 & -x \\
					-x & 1
				\end{matrix}\right),
	\eeq
	which makes sense for any commutative ring $k$. This metric can also be written in the alternative form
	\beq
		\label{poly_metric_ex}
		\begin{aligned}
			g & = dx \otimes dx + x \, (dx \otimes dy + dy \otimes dx) + (1 + x^2) \, dy \otimes dy \\
				& = (1 + x)^2 \, (dx + dy) \otimes (dx + dy) + dy \otimes dy.
		\end{aligned}
	\eeq
	\end{enumerate}
\end{ex}

\begin{rem}
	\label{signature}
	Let us comment on which additional conditions one might want to generally impose on a metric as algebraic analogues of a signature condition.
	For example, the metric~\eqref{poly_metric_ex} should arguably be considered Riemannian for any $k$, since it is a sum of squares, which indeed enforces Riemannian signature for $k = \R$.
	As a weaker condition, one could merely require that $g(v,v) \in \ZLA$ should be a sum of squares in $\ZLA$ for every $v \in \D$. In the manifold case $C^\infty(M)$, this would in particular restrict $g(v,v)$ to be a nonnegative smooth function. But unfortunately the converse does not hold: there are nonnegative smooth functions on $\R^4$ that cannot be written as sums of squares of other smooth functions~\cite{BBCFP}.
	It follows that requiring every function of the form $g(v,v)$ to be a sum of squares is sufficient but not necessary for positive semidefiniteness of $g$, and thereby indicates that this is still too strong as an algebraic signature condition.\footnote{This is reminiscent of the situation for polynomials over $\R$, where also not every nonnegative polynomial is a sum of squares~\cite{marshall}.}
	We therefore do not currently have a satisfactory algebraic generalization of a signature condition.
\end{rem}

It is a standard textbook fact that every smooth manifold admits a (Riemannian) metric.
The situation for algebraifolds is less clear.

\begin{prob}
	Does every algebraifold admit a metric?
\end{prob}

For example, if an algebraifold $\ZLA$ is such that $\D$ is a free module, then we can construct a metric by using the standard inner product with respect to any basis of $\D$.
In general, clearly a necessary condition for the existence of a metric is the existence of an $\ZLA$-module isomorphism $\D_\ZLA \cong \D_\ZLA^*$, and this observation may help in the search for a potential counterexample.

Moving on to the next concept, the \newterm{Levi--Civita connection} of a metric $g$ on an algebraifold can be defined in terms of \emph{Koszul's formula}. This means that for $u,v \in \D$, the covariant derivative $\nabla_u v \in \D$ is defined implicitly by its inner product with any $w \in \D$, which takes the form
\begin{align}
	\begin{split}
		\label{koszul}
			g(\nabla_u v, w) = \frac{1}{2} \bigg(	& u \left( g(v,w) \right) + v \left( g(w,u) \right) - w \left( g(u,v) \right)		\\
						& + g([u,v],w) - g([v,w],u) + g([w,u],v) \bigg).
	\end{split}
\end{align}
Note that due to the factor of $\frac{1}{2}$, this now only makes sense if $2$ is invertible in $k$, which we assume to be the case:

\begin{ass}
	\label{two_inv}
	For the remainder of this section, $\frac{1}{2} \in k$.
\end{ass}

A brief computation shows that the right-hand side of~\eqref{koszul} is indeed linear in $w$, which implies that we obtain a well-defined element $\nabla_u v \in \D$ upon tensoring by $g^{-1}$ and contracting in order to solve for $\nabla_u v$.
The linearity in $u$ similarly follows by a simple computation, as does the Leibniz rule in $v$. The standard argument as used in the manifold case also shows that the Levi--Civita connection is the unique torsion-free connection for which the covariant derivative of the metric vanishes,
\[
	\nabla_u g = 0 \qquad \forall u \in \D.
\]
For more detail on the derivation of Koszul's formula and the uniqueness in the algebraic setting, we also refer to the treatment of the Levi--Civita connection in the context of the \emph{Rinehart spaces} of Pessers and van der Veken~\cite[Theorem~30]{PV}, which applies in the algebraifold setting without changes.

\begin{ex}
	\label{levi_civita_christoffel}
	\label{greek indices?}
	For the polynomial ring $k[x_1,\ldots,x_n]$, \Cref{christoffel} implies that we can express the Levi--Civita connection in terms of a tensor $\Gamma$ of Christoffel symbols. Koszul's formula~\eqref{koszul} then shows that this tensor's components have the usual form
	\beq
		\label{christoffel_usual}
		\Gamma^k_{ij} = \frac{1}{2} g^{kl} \left( \frac{\partial g_{il}}{\partial x^j} + \frac{\partial g_{jl}}{\partial x^i} - \frac{\partial g_{ij}}{\partial x^l} \right),
	\eeq
	using the fact that the coordinate vector fields $\frac{\partial}{\partial x^i}$ all commute.
\end{ex}

\subsection{Riemann and Ricci curvature}

As in the manifold case, the Riemann tensor of a metric can now be defined as the curvature of its Levi--Civita connection. The \newterm{Ricci tensor} $\Ric$ is the contraction of the Riemann tensor~\Cref{curvature_algebraic} given by contracting the first argument with its output. In terms of dual bases as in \Cref{dual_bases_algebraifold}, this reads
\begin{equation}
	\label{ricci}
	\Ric(v,w) \coloneqq \sum_{i=1}^n \left( R(u_i, v) w \right) (a_i).
\end{equation}
The \newterm{Ricci scalar} $S$ is the (unique) contraction of the Ricci tensor with the inverse of the metric,
\[
	S \coloneqq \sum_{i,j=1}^n g^{-1}(da_i, da_j) \, \Ric(u_i, u_j).
\]
Both of these definitions are abstract versions of the standard expressions, extending the latter in the obvious way from the manifold setting to the algebraifold setting.

\section{The category of algebraifolds and the problem of products}
\label{cats}

Of course, an important concept in differential geometry is the notion of smooth map.
What is the analogue of that for algebraifolds?
As usual when considering algebraic formulations of geometric concepts, the direction of the arrows reverses, and we therefore start by considering maps in the algebraic direction.
The developments of this section will be relevant in particular for the consideration of geodesics in \Cref{geodesics}.

\subsection{Algebraifold homomorphisms}

The relevant definition, where $k$ is again any commutative ring, is a variant of a definition due to Beggs and Majid~\cite[(1.1)]{BM}.

\begin{defn}
	\label{algebraifold_homomorphism}
	For $k$-algebraifolds $\ZLA$ and $\ZLB$ in standard form, an \newterm{algebraifold homomorphism} is a $k$-algebra homomorphism $\varphi : \ZLA \to \ZLB$ for which there is an $\ZLA$-module map $\Omega_{\varphi} : \Omega_{\ZLA} \to \varphi^* \Omega_{\ZLB}$ such that the diagram
	\begin{equation}
		\label{pullback_1form}
		\begin{tikzcd}
			\ZLA \ar[r, "\varphi"] \ar[d, swap, "d"] & \ZLB \ar[d, "d"] \\
			\Omega_{\ZLA} \ar[r, "\Omega_{\varphi}"] & \varphi^* \Omega_{\ZLB}
		\end{tikzcd}
	\end{equation}
	commutes.
\end{defn}

Here, $\varphi^* \Omega_{\ZLB}$ denotes $\Omega_{\ZLB}$, considered as an $\ZLA$-module via restriction of scalars (\Cref{sec_extension}).

\begin{ex}
	\label{algebraifold_homomorphism_manifold}
	In the manifold case, composition with any smooth map $f : M \to N$ induces an algebraifold homomorphism
	\[
		C^\infty(f) \: : \: C^\infty(N) \longrightarrow C^\infty(M).
	\]
	Indeed as the associated map
	\[
		\Omega_{C^\infty(f)} \: : \: \Omega_{C^\infty(M)} \longrightarrow C^\infty(f)^* \Omega_{C^\infty(N)},
	\]
	we can take the pullback of $1$-forms on $N$ to $1$-forms on $M$, since this pullback operation has the defining property that
	\[
		 \Omega_{C^\infty(f)} (d g) = d(g \circ f) 
	\]
	holds for all $g \in C^\infty(N)$, which expresses the commutativity of~\eqref{pullback_1form}.
\end{ex}

So in general, the map $\Omega_{\varphi}$ from \Cref{algebraifold_homomorphism} encodes the algebraic generalization of the pullback of $1$-forms along a smooth map.
The commutativity of~\eqref{pullback_1form} amounts to the equation
\begin{equation}
	\label{pullback_1form_req}
	\Omega_{\varphi}(d a) = d (\varphi(a)) \qquad a \in \ZLA.
\end{equation}
Since $\Omega_{\ZLA}$   is generated by the differentials $d a$ for $a \in \ZLA$, it is clear that $\Omega_{\varphi}$ is unique if it exists.
Hence there is no need to consider $\Omega_{\varphi}$ as part of the data of an algebraifold homomorphism, and postulating its existence in the definition is enough.

In fact, we do not know of any single example of an algebra homomorphism $\varphi : \ZLA \to \ZLB$ between algebraifolds for which $\Omega_{\ZLA}$ does \emph{not} exist.
Therefore it is conceivable that its existence is automatic.
Indeed as we saw in \Cref{algebraifold_homomorphism_manifold}, this is what happens case in the manifold case, where the homomorphisms $C^\infty(N) \to C^\infty(M)$ are exactly the smooth maps $M \to N$~\cite[Corollary~35.10]{KMS}.
A general proof is not obvious: the universal property of $\Omega_{\ZLA}$ from \Cref{kahler_algebraifold} does not apply because $\varphi^* \Omega_{\ZLB}$ need not be fgp.

\begin{prob}
	\label{existence_pullback}
	Is the existence of $\Omega_{\varphi}$ in \Cref{algebraifold_homomorphism} automatic?
\end{prob}

For any algebraifold homomorphism $\varphi : \ZLA \to \ZLB$, we can also find an explicit expression for $\Omega_{\varphi}$ using dual bases.

\begin{lem}
	\label{lem_explicit_pullback}
	In terms of dual bases as in \Cref{dual_bases_algebraifold}, we can write $\Omega_\varphi$ explicitly as
	\begin{equation}
		\label{explicit_pullback}
		\Omega_{\varphi}(\xi) = \sum_{i=1}^n \varphi(u_i(\xi)) \, d \varphi(a_i)
	\end{equation}
	for all $\xi \in \Omega_{\ZLA}$.
\end{lem}

This formula does not immediately resolve \Cref{existence_pullback}, since using~\eqref{explicit_pullback} as the definition of $\Omega_{\varphi}$ leaves open whether the relevant condition~\eqref{pullback_1form_req} is satisfied.

\begin{proof}
	This can be seen from~\eqref{pullback_1form_req} together with $\xi = \sum_{i=1}^n u_i(\xi) \, d a_i$ and the $\ZLA$-linearity of $\Omega_{\varphi}$.
\end{proof}

In terms of the adjunction~\eqref{adjunction_scalars} between restriction and extension of scalars, we can also consider $\Omega_{\varphi}$ as a map $\varphi_* \Omega_{\ZLA} \to \Omega_{\ZLB}$, uniquely determined by the requirement that it should map
\begin{equation}
	\label{pullback_1form_req2}
	1_\ZLB \otimes d a \longmapsto d \varphi(a).
\end{equation}
This description is interesting insofar as it can be most easily instantiated and reformulated further.

\begin{ex}
	\label{tangent_action}
	In the manifold case with $\varphi = C^\infty(f)$ as in \Cref{algebraifold_homomorphism_manifold}, the $C^\infty(M)$-module $\varphi_* \Omega_{C^\infty(M)}$ corresponds to the pullback bundle of the cotangent bundle $f^* (T^* N)$.
	In this situation, the extension of scalars $\varphi_* \Omega_{C^\infty(N)}$ is (naturally isomorphic to) the module of sections of the pullback bundle $f^*(T^* N)$ by \Cref{pullback}, and the map $\varphi_* \Omega_\ZLA \to \Omega_\ZLB$ corresponds to a vector bundle map $f^*(T^* M) \to T^* M$.
	The requirement~\eqref{pullback_1form_req2} characterizes this as the usual pullback of $1$-forms via \Cref{pullback_mfd}.
\end{ex}

Let us now massage $\Omega_\varphi$ a bit further in order to see how it induces a map between modules of derivations, which will turn out to be the algebraic generalization of the differential of a smooth map..
Using first the assumption that $\D_{\ZLA}$ is fgp, we can apply the duality between $\varphi_* \D_{\ZLA}$ and $\varphi_* \Omega_{\ZLA}$ from \Cref{extension_dual_eq} to turn~\eqref{pullback_1form_req2} into an element
\[
	d\varphi \,\in\, \Omega_{\ZLB} \otimes_{\ZLA} \varphi_* \D_{\ZLA}.
\]
Concretely, in terms of dual bases $(da_i)$ of $\Omega_{\ZLA}$ and $(u_i)$ of $\D_{\ZLA}$ as in \Cref{dual_bases_algebraifold}, we can use the construction~\eqref{dual_basis_adjunct} to write this element explicitly as
\begin{equation}
	\label{differential}
	d\varphi = \sum_{i=1}^n d \varphi(a_i) \otimes (1_\ZLB \otimes u_i).
\end{equation}
Since $\Omega_{\ZLB}$ is also in duality with $\D_{\ZLB}$, we furthermore can turn $d \varphi$ into a $\ZLB$-linear map
\[
	\D_\varphi \: : \: \D_{\ZLB} \longrightarrow \varphi_* \D_{\ZLA},
\]
and this is what we call the \newterm{differential} of $\varphi$.

\begin{lem}
	In terms of dual bases as in \Cref{dual_bases_algebraifold}, we can write $\D_\varphi$ explicitly as
	\begin{equation}
		\label{explicit_differential}
		\D_\varphi(w) = \sum_{i=1}^n w(\varphi(a_i)) \otimes u_i
	\end{equation}
	for all $w \in \D_{\ZLB}$.
\end{lem}

\begin{proof}
	Using $d \varphi$ in the form~\eqref{differential} together with~\eqref{dual_basis_adjunct_inv} shows that $\D_\varphi$ is given by
	\[
		\D_\varphi(w) = \sum_{i=1}^n d \varphi(a_i)(w) \, (1_\ZLB \otimes u_i) = \sum_{i=1}^n w(\varphi(a_i)) \, (1_\ZLB \otimes u_i) = \sum_{i=1}^n w(\varphi(a_i)) \otimes u_i,
	\]
	as was to be shown.
\end{proof}

By construction, the module homomorphisms $\Omega_\varphi$ and $\D_\varphi$ are each other's adjoints with respect to the canonical evaluate pairings, in the sense that we have
\begin{equation}
	\label{adjointness_pullback}
	\Omega_\varphi(\xi)(w) = (1_\ZLB \otimes \xi)(\D_\varphi(w)) \qquad \forall \xi \in \Omega_\ZLA, \, w \in \D_\ZLB,
\end{equation}
and this determines each map in terms of the other.
On the right-hand side, this uses the canonical identification of $\varphi_* \Omega_\ZLA$ as the dual of $\varphi_* \D_\ZLA$.

\begin{ex}
	We now argue that in the manifold case, the map
	\[
		\D_{C^\infty(f)} \: : \: \D_{C^\infty(M)} \longrightarrow C^\infty(f)_* \D_{C^\infty(N)}
	\]
	corresponds to the usual differential of a smooth map $f : M \to N$ as acting on tangent vectors.
	Indeed this action on tangent vectors turns every vector field on $M$ into a section of the pullback bundle $f^*(TN)$, and is therefore already of the required type by \Cref{pullback}.
	This ordinary differential is characterized by the condition that pairing a vector field $w$ on $M$ with the pullback of a $1$-form $\xi$ on $N$ produces the element of $C^\infty(M)$ equal to the pairing of the associated sections of $f^*(T N)$ and $f^*(T^* N)$.
	Since this is exactly what the adjointness relation~\eqref{adjointness_pullback} expresses, we conclude that $\D_\varphi$ indeed recovers the usual differential in the manifold case.

	For $N = \R^n$, the bundle $f^* (T N)$ is the trivial bundle of rank $n$ on $M$.
	Using the dual coordinate bases as in~\Cref{Rn_dual_bases} for $C^\infty(\R^n)$, we have $\varphi(a_i) = f_i$.
	Hence~\eqref{explicit_differential} becomes in this case
	\[
		\D_{C^\infty(f)}(w) = \sum_{i=1}^n w(f_i) \otimes \frac{\partial}{\partial x_i},
	\]
	which indeed matches the obvious coordinate expression of the Jacobian of $f$ acting on $w$ at every point.
\end{ex}

\begin{lem}
	\label{differential_char}
	For all $w \in \D_{\ZLB}$ and $a \in \ZLA$, we have
	\[
		\D_{\varphi}(w)_{(1)} \otimes \D_{\varphi}(w)_{(2)}(a) = w(\varphi(a)),
	\]
	and $\D_{\varphi} : \D_{\ZLB} \to \varphi_* \D_{\ZLA}$ is the only $\ZLB$-linear map satisfying this property.
\end{lem}

Here, the left-hand side is an element of
\[
	\varphi_* \ZLA = \ZLB \otimes_{\ZLA} \ZLA \cong \ZLB,
\]
and we suppress the isomorphism from the notation.
have suppressed the canonical isomorphism  from the notation.
We also have employed sumless Sweedler notation for $\D_\varphi(w)$.

\begin{proof}
	We have
	\begin{align*}
		\D_{\varphi}(w)_{(1)} \otimes \D_{\varphi}(w)_{(2)}(a) & = (1_\ZLB \otimes da) \D_{\varphi}(w) \\
								       & = \Omega_\varphi(da)(w) \\
								       & = d(\varphi(a))(w) \\
								       & = w(\varphi(a)),
	\end{align*}
	where the second step uses~\eqref{adjointness_pullback}.
	The uniqueness claim follows by the second expression together with the fact that the elements of the form $1_{\ZLB} \otimes da$ generate $\varphi_* \Omega_{\ZLA}$ as a $\ZLB$-module.
\end{proof}

This ends our discussion of the action of an algebraifold homomorphism on derivations and $1$-forms.
One further observation that will be useful in \Cref{geodesics} is that connections can be transported along homomorphisms by extension of scalars.

\begin{lem}
	\label{connection_transport}
	Let $\varphi : \ZLA \to \ZLB$ be a homomorphism of $k$-algebraifolds, and let
	\[
		\nabla \: : \: \D_{\ZLA} \times \Mod \longrightarrow \Mod
	\]
	be a connection on an $\ZLA$-module $\Mod$.
	Then there is an induced connection $\varphi_* \nabla$ on $\varphi_* \Mod$ uniquely determined by
	\begin{equation}
		\label{transported_connection}
		(\varphi_* \nabla)_w (1_\ZLB \otimes x) = \D_{\varphi}(w)_{(1)} \otimes \nabla_{\D_{\varphi}(w)_{(2)}}(x)
	\end{equation}
	for all $w \in \D_{\ZLB}$ and $x \in \Mod$.
\end{lem}

\begin{proof}
	We consider the more general formula
	\begin{equation}
		\label{transported_connection2}
		(\varphi_* \nabla)_w (b \otimes x) = b \D_{\varphi}(w)_{(1)} \otimes \nabla_{\D_{\varphi}(w)_{(2)}}(x) + w(b) \otimes x,
	\end{equation}
	of which~\eqref{transported_connection} is the special case $b = 1_\ZLB$.
	Evaluating it on an element of the form $\varphi(a) b \in \ZLB$ in place of $b$ produces the same expression multiplied by $\varphi(a)$ and plus an extra term given by $w(\varphi(a)) b \otimes x$.
	Plugging in $a x$ in place of $x$ similarly produces the extra term
	\[
		b \D_{\varphi}(w)_{(1)} \otimes \D_{\varphi}(w)_{(2)}(a) x = w(\varphi(a)) b \otimes x,
	\]
	where the equality holds by \Cref{differential_char}.
	Hence $\varphi_* \nabla$ is well-defined in the second argument $\varphi_* \Mod = \ZLB \otimes_{\ZLA} \Mod$.

	To see that $\varphi_* \nabla$ is a connection, we need to check that it is $\ZLB$-linear in the first argument and satisfies the Leibniz rule in the second.
	The $\ZLB$-linearity in the first argument is immediate.
	The Leibniz rule in the second argument with respect to multiplication by scalars from $\ZLB$ follows by direct calculation as in the well-definedness argument, with $\varphi(a)$ replaced by a generic element of $\ZLB$.

	The uniqueness claim is clear as~\eqref{transported_connection2} is a necessary consequence of~\eqref{transported_connection} by the Leibniz rule in the second argument.
\end{proof}

It is know that \Cref{connection_transport} recovers the standard notion of pullback connection in the manifold context.
For example in the picture of \Cref{connection_form}, where a connections is $1$-form-valued map, a version of our \Cref{connection_transport} specialized to the manifold case is given e.g.~in \cite[Lemma~17.10]{MT}.
Nevertheless, we briefly sketch the correspondence to a point-based description as follows.

\begin{ex}
	\label{pullback_connection}
	For a smooth map $f : M \to N$ between smooth manifolds, and a vector bundle with connection $\nabla$ on $N$, it is well-known\footnote{It seems difficult to find a pedagogical exposition of this construction including~\eqref{pb_conn_formula} in published literature, but the reader will have no difficulty reading about it on the web, e.g.~at \href{https://mathoverflow.net/questions/49272/pull-back-connection}{mathoverflow.net/questions/49272/pull-back-connection}.} that there is an induced connection $f^* \nabla$ on the pullback bundle uniquely determined by
	\begin{equation}
		\label{pb_conn_formula}
		(f^* \nabla)_w (f^* x) = f^* (\nabla_{f_* w} x),
	\end{equation}
	where $x$ is any section of the original bundle on $N$ and $w$ is any tangent vector in $M$.
	This can be seen to be the pointwise version of~\eqref{transported_connection}, based on the isomorphism between the extension of scalars module and the sections of the pullback bundle.

\end{ex}

\subsection{The category of algebraifolds}

If $\varphi : \ZLA \to \ZLB$ and $\psi : \ZLB \to \ZLC$ are algebraifold homomorphisms, then the composition $\psi \varphi : \ZLB \to \ZLC$ is again an algebraifold homomorphism thanks to the obvious identity $(\psi \varphi)^* \Omega_{\ZLC} = \psi^* (\varphi^* \Omega_{\ZLC})$.

As usual when considering algebraic formulations of geometric concepts, the direction of the arrows reverses when thinking of the algebraic structures as geometric entities.
It seems useful to have this reversal reflected in the terminology.

\begin{defn}
	For $k$-algebraifolds $\ZLA$ and $\ZLB$, an \newterm{algebraifold map} $\varphi^\op : \ZLB \rightsquigarrow \ZLA$ is the formal dual of an algebraifold homomorphism $\varphi : \ZLA \to \ZLB$.
\end{defn}

Of course, algebraifold maps compose in the opposite direction to algebraifold homomorphisms, and we thus obtain the following category.

\begin{defn}
	Given a commutative ring $k$, the \newterm{category of $k$-algebraifolds} $\Afd{k}$ has:
	\begin{itemize}
		\item $k$-algebraifolds in standard form as objects,
		\item Algebraifold maps as morphisms.
	\end{itemize} 
\end{defn}

Mapping a connected smooth manifold to its $\R$-algebra of smooth functions and a smooth map $M \to N$ to the induced $\R$-algebra homomorphism $C^\infty(N) \to C^\infty(M)$ defines a functor
\beq
	\label{man_to_afd}
	\Man \longrightarrow \Afd{\R}.
\eeq
In fact, this functor is fully faithful, meaning that it establishes a bijection between smooth maps $M \to N$ and $\R$-algebra homomorphisms $C^\infty(N) \to C^\infty(M)$.\footnote{As we learned from Eugene Lerman, this result goes back to the 1952 thesis of Pursell~\cite[Chapter~8]{pursell}, which already contains its essential ingredients. We refer to~\cite[Corollary~35.10]{KMS} for a textbook account. It is also worth noting that every mere ring homomorphism $C^\infty(N) \to C^\infty(M)$ is automatically $\R$-linear, as one can see e.g.~by first showing that it automatically preserves the pointwise order on functions, so that $R$-linearity follows by $\Q$-linearity.}
The framework of \emph{stuff, structure and property}~\cite{BS} allows us to phrase this rigorously as follows.

\begin{slog}
	A smooth manifold is an $\R$-algebraifold with extra property (but no extra structure or stuff).
\end{slog}

It then becomes an interesting question how to recognize those $\R$-algebraifold $\ZLA$ that correspond to manifolds, i.e.~for which there is an algebra isomorphism $\ZLA \cong C^\infty(M)$ for some smooth manifold $M$.
One such characterization can be obtained by combining the results of~\cite{KKM} and~\cite{MV}\footnote{We thank Igor Khavkine for pointing this out to us.}; we do not spell this out here as doing so would require introducing the language of $C^\infty$-rings first.
As far as we are aware, a more explicit characterization phrased in terms of $\R$-algebra structure only is not known.

\subsection{The problem of products}

Some of the most important constructions of manifolds can be expressed nicely in terms of universal properties in the category of manifolds $\Man$, by which we mean the category with connected smooth manifolds as objects and smooth maps as morphisms.
For example, the product $M \times N$ of manifolds $M$ and $N$ is the categorical product in $\Man$.
How does this work out for algebraifolds?

\begin{prob}
	\label{prods}
	\begin{enumerate}
		\item\label{afd_prods} For a given commutative ring $k$, does the category $\Afd{k}$ have products?
		\item If $\Afd{\R}$ has products, does the functor~\eqref{man_to_afd} map products of manifolds to products of algebraifolds?
	\end{enumerate}
\end{prob}

For~\ref{afd_prods}, we know that products do exist at least in certain cases.
Indeed one obvious candidate for a product of $k$-algebraifolds $\ZLA$ and $\ZLB$ is the tensor product $\ZLA \otimes_k \ZLB$, where $\otimes_k$ denotes the usual tensor product (or equivalently coproduct) of commutative $k$-algebras.
We do not know under which conditions this is an algebraifold again in general, but here is one case in which it is true.

\begin{prop}
	Let $k$ be a field, and suppose that $\ZLA$ and $\ZLB$ are algebraifolds with one of them finitely generated as a $k$-algebra.
	Then $\ZLA \otimes_k \ZLB$ is an algebraifold again, and it is the product of $\ZLA$ and $\ZLB$ in $\Afd{k}$.
\end{prop}

\begin{proof}
	Under the present assumptions, there is a direct sum decomposition~\cite[Theorem~2.8]{azam}
	\begin{equation}
		\label{dsum}
		\D_{\ZLA \otimes_k \ZLB} \cong \left( \D_{\ZLA} \otimes_k \ZLB \right) \oplus \left( \ZLA \otimes_k \D_{\ZLB} \right),
	\end{equation}
	where both summands are $\ZLA \otimes_k \ZLB$-modules in the obvious way.
	The first claim now follows as both summands on the right-hand side are fgp over $\ZLA \otimes_k \ZLB$.

	For the universal property, we show that $\ZLA \otimes_k \ZLB$ is the coproduct of $\ZLA$ and $\ZLB$ in the category of $k$-algebraifolds and algebraifold homomorphisms.
	To see this, note first that the isomorphism~\eqref{dsum} dualizes to\footnote{To see that $\D_\ZLA \otimes_k \ZLB$ dualizes to $\Omega_\ZLA \otimes_k \ZLB$, apply~\Cref{extension_dual} with respect to the canonical homomorphism $\ZLA \to \ZLA \otimes_k \ZLB$.}
	\[
		\Omega_{\ZLA \otimes_k \ZLB} \cong \left( \Omega_{\ZLA} \otimes_k \ZLB \right) \oplus \left( \ZLA \otimes_k \Omega_{\ZLB} \right),
	\]
	Together with
	\begin{equation}
		\label{differential_sum}
		d(a \otimes b) = (da \otimes b) \oplus (a \otimes db),
	\end{equation}
	it follows that the usual homomorphisms $\ZLA \to \ZLA \otimes_k \ZLB$ and $\ZLB \to \ZLA \otimes_k \ZLB$ are algebraifold homomorphisms, since the corresponding action on $1$-forms exists.
	To show the universal property, it remains to be proven that if $\varphi : \ZLA \to \ZLC$ and $\psi : \ZLB \to \ZLC$ are algebraifold homomorphisms, then so is the induced $k$-algebra homomorphism $\ZLA \otimes_k \ZLB \to \ZLC$.
	But this is also clear by~\eqref{differential_sum}.
\end{proof}

In general, it is conceivable that products in $\Afd{k}$ exist and can be computed from the algebraic tensor product in a universal way.
This would follow directly from a positive answer to the following, where $\Alg{k}$ denotes the category of commutative $k$-algebras and algebra homomorphisms.

\begin{qstn}
	Does the inclusion functor $\Afd{k}^\op \hookrightarrow \Alg{k}$ have a left adjoint?
\end{qstn}

If such an adjoint exists, then the product of two algebraifolds $\ZLA$ and $\ZLB$ can be computed as the left adjoint applied to the algebraic tensor product $\ZLA \otimes_k \ZLB$, since this is the coproduct in $\Alg{k}$ and left adjoint functors preserve colimits.

\section{Formal lines and geodesics}
\label{geodesics}

The goal of this section is to show that even the concept of geodesic has a sensible algebraic formulation.
This is an important development for algebraic approaches to differential geometry, given that the standard definition is point-based, which makes the translation into algebraic language less obvious than that of tensor calculus and connections.

\subsection{Formal lines}

Since geodesics are certain curves in a manifold, generalizing them to our setting first of all requires an algebraization of the notion of curve, to be defined as an algebraifold map from a suitable kind of ``line'' to another algebraifold.
The relevant notion of line is as follows, where $k$ is still any commutative ring.

\begin{defn}
	\label{formal_line}
	A $k$-algebraifold $\ZLL$ in standard form is a \newterm{formal line} if the following hold:
	\begin{enumerate}
		\item $\D_\ZLL$, or equivalently $\Omega_\ZLL$, is a free $\ZLL$-module of rank one.
		\item\label{formal_line_surj}
			Some $\partial \in \D_\ZLL$ is surjective as a map $\ZLL \to \ZLL$.
	\end{enumerate}
\end{defn}

\begin{ex}
	\label{poly_ring_formal_line}
	For any commutative ring $k$ with $\Q \subseteq k$, the $k$-algebraifold $k[t]$ is a formal line.
	In particular, with $\partial : k[t] \to k[t]$ the usual derivative map, condition~\ref{formal_line_surj} holds as every polynomial is the derivative of its antiderivative.

	On the other hand, if $k$ is a field of characteristic $\ell > 0$, then $k[t]$ is not a formal line, not even over its ring of constants (\Cref{poly_ring2}), since for example the monomial $x^{\ell-1}$ is not a derivative of any other polynomial, and taking $\partial$ to be any derivation other than standard differentiation does not help.
	In fact, it is unclear to us whether a formal line over a field of positive characteristic exists at all.
\end{ex}

\begin{ex}
	\label{smooth_formal_line}
	Over $k = \R$, the algebraifold $C^\infty(\R)$ is a formal line, where $\partial$ can be taken to be any vector field that does not vanish anywhere.
\end{ex}

\begin{rem}
	$C^\infty(S^1)$ is not a formal line over $\R$. 
	The module of derivations $\D_{C^\infty(S^1)}$ is generated by any vector field that does not vanish anywhere, and up to automorphisms of the algebra, this can be taken to be the standard vector field $\partial_\theta$ on $S^1$ with $\theta$ being the angle variable.
	This indeed makes $\D_{C^\infty(S^1)}$ free of rank one, but~\ref{formal_line_surj} fails: only those smooth functions which have zero integral with respect to $d \theta$ can appear as a derivative.
\end{rem}

So intuitively, condition~\ref{formal_line_surj} has a topological flavour along the lines of vanishing first de Rham cohomology.\footnote{This can possibly be made into a general precise statement based on a suitable definition of de Rham cohomology for algebraifolds.}

\begin{lem}
	\label{formal_line_antiderivative}
	\begin{enumerate}
		\item There is $t \in \ZLL$ such that $\partial t = 1$.
		\item For every $a \in \ZLL$, there is $\int a \in \ZLL$ with
			\[
				\partial \int a = a,
			\]
			and this $\int a$ is unique up to a constant in $k$.
	\end{enumerate}
\end{lem}

\begin{proof}
	\begin{enumerate}
		\item This is clear by \Cref{formal_line}\ref{formal_line_surj}.
		\item The existence holds again by~\ref{formal_line_surj}.
			For the uniqueness, it is enough to note that if $\partial b = 0$ for $b \in \ZLL$, then $b$ belongs to the ring of constants.
			This is because every $\partial$ generates $\D_\ZLL$, so that $\partial b = 0$ indeed implies $v(b) = 0$ for all $v \in \D_\ZLL$.
			\qedhere
	\end{enumerate}
\end{proof}

As the antiderivative of $1_\ZLL$, such an element $t$ plays the role of the variable $t$ in \Cref{poly_ring_formal_line} and the identity function in \Cref{smooth_formal_line}.

\begin{rem}
	If $\Q \subseteq k$, then every formal line $\ZLL$ contains $k[t]$ as a subalgebra, where again $t = \int 1$.
	Indeed the induced algebra homomorphism $k[t] \to \ZLL$ is injective, as can be seen by differentiating a given polynomial degree many times.
\end{rem}

Given any $k$-algebraifold $\ZLA$ in addition to a formal line, we can now consider a curve in $\ZLA$ to be an algebraifold map $\ZLL \rightsquigarrow \ZLA$, or equivalently an algebraifold homomorphism $\ZLA \to \ZLL$.
As per the following example, the choice of formal line determines which kind of curves are considered.

\begin{ex}
	For $\ZLA = \R[x_1, \ldots, x_n]$, the algebraifold maps $\ZLA \rightsquigarrow \R[t]$ correspond to the polynomial maps $\R \to \R^n$.	
	The algebraifold maps $\ZLA \rightsquigarrow C^\infty(\R)$ correspond to the \emph{smooth} maps $\R \to \R^n$.

	Both of these claims are straightforward to see by using that $\R[x_1, \ldots, x_n]$ is the free $\R$-algebra on $n$ generators, so that the $\R$-algebra homomorphisms $\R[x_1, \ldots, x_n] \to \R$ are in bijection with the $n$-tuples of elements of $\R$, and the extra condition of \Cref{algebraifold_homomorphism} is clearly satisfied.
\end{ex}

\begin{ex}
	If $M$ is a smooth manifold, then the algebraifold maps $C^\infty(M) \rightsquigarrow C^\infty(\R)$ correspond to the smooth curves $\R \to M$ by the fact that the functor $\Man \to \Afd{\R}$ is fully faithful~\eqref{man_to_afd}.

	On the other hand, the only algebraifold maps $\R[t] \rightsquigarrow C^\infty(M)$ are those that factor across $\R$, meaning that they correspond to \emph{constant} curves in $M$.
	This is because the only $\R$-algebra homomorphisms $C^\infty(M) \to \R[t]$ are those that factor across a point evaluation homomorphism $C^\infty(M) \to \R$.\footnote{To see this, compose with the embedding $\R[t] \hookrightarrow C^\infty(\R)$ to note that every homomorphism $C^\infty(M) \to \R[t]$ is given by restriction along a smooth map $\R \to M$, and such a map must be constant since otherwise there will be smooth functions on $M$ that do not restrict to polynomial functions on $\R$.}
	We think of this as indicating that there is no notion of ``polynomial curve'' in $M$.
\end{ex}

Let us now turn to geodesics.
In the manifold setting, a geodesic is a smooth curve $\gamma : \R \to M$ whose tangent vectors are covariant constant along the curve.
In terms of the notion of pullback connection from~\Cref{pullback_connection}, this amounts to saying that the induced connection on the bundle $\gamma^* TM$ should be such that the pulled back tangent vectors of $\gamma$ are covariant constant with respect to the pulled back connection.
This formulation generalizes to the following notion of geodesic associated to any formal line.

\begin{defn}
	\label{geodesic}
	Let $\ZLA$ be any $k$-algebraifold equipped with a connection $\nabla$ and $\ZLL$ a formal line over $k$.
	Then an \newterm{$\ZLL$-geodesic in $\ZLA$} is an algebraifold map $\varphi^\op : \ZLL \rightsquigarrow \ZLA$ such that
	\begin{equation}
		\label{geodesic_equation}
		(\varphi_* \nabla)_\partial (\D_\varphi \partial) = 0.
	\end{equation}
\end{defn}

Here, we use the connection $\varphi_* \nabla$ as defined by \Cref{connection_transport}, which is the connection $\nabla$ transported along $\varphi$ to the formal line $\ZLL$.

\begin{ex}
	For a smooth manifold $M$ and smooth curve $\gamma \colon \R \to M$, the associated algebraifold homomorphism
	\[
		C^\infty(\gamma) \: : \: C^\infty(M) \to C^\infty(\R)
	\]
	satisfies our \eqref{geodesic_equation} if and only if $\gamma$ is a geodesic in the standard sense.
	Indeed $\D_{C^\infty(\gamma)} \partial$ is precisely the collection of velocity vectors of $\gamma$, considered as a section of the pullback bundle $\gamma^* TM$, and by~\eqref{pb_conn_formula} our equation states that this vector field should be covariant constant with respect to the pullback connection on $\R$.
\end{ex}

\begin{rem}
	Based on the differential of an algebraifold map in the form~\eqref{differential}, we expect that a general definition of \emph{harmonic map} between algebraifolds can be given, but the details remain to be worked out.
\end{rem}

\section{Algebraifolds in general relativity}
\label{GR}

\subsection{The Einstein field equation}

Based on the notions of metric, connection and curvature from \Cref{sec:diff_geom}, it is straightforward to also write down the Einstein field equation of general relativity in its usual form, now amounting to an equation for an unknown metric on a given algebraifold $\ZLA$ in standard form. In terms of the Ricci tensor and scalar considered at~\eqref{ricci}, the field equations take the form
\beq
	\label{efe}
	\Ric - \frac{1}{2} S g + \Lambda g = \kappa T,
\eeq
where now both the gravitational constant $\kappa$ and the cosmological constant $\Lambda$ are assumed to be constants, which in our context means fixed invertible elements of $k$, and $T$ is the stress-energy tensor, also a $(0,2)$-tensor as usual, for which nay of the standard forms in terms of matter fields can be assumed.\footnote{At least for bosonic matter, considering fermionic matter will first require a generalization of spinor fields to the algebraifold setting.}

Following Geroch's idea~\cite{geroch}, we may now call \newterm{Einstein $k$-algebra}, or just \newterm{Einstein algebra}, any $k$-algebraifold together with a metric which satisfies the Einstein field equation.
The remainder of the paper will sketch some examples of Einstein algebras that are not manifolds, illustrating the potential relevance of our framework to general relativity.
An interesting example that we will not discuss further is given by general relativity with generalized metrics in the sense of \Cref{colombeau}, which allows for certain singularities~\cite{GKOS}.

\subsection{Function fields}

Physicists, especially those who do generally not work with full mathematical rigour, often like to apply algebraic operations without worrying about whether this makes sense in the underlying mathematical structure. 
In particular, they may take the inverse $f^{-1}$ of a function $f$ without ensuring that $f$ does not have any zeroes.
The problem is that this inverse $f^{-1}$ does not exist as a smooth function on a manifold $M$ in case that $f$ has a zero.
Even worse, $C^\infty(M)$ has lots of zero divisors, so that this problem cannot be rectified by embedding into a field.
Hence a mathematically inclined physicist will object that considering $f^{-1}$ does not make sense, unless it is first ensured that $f$ does not have a zero.

However, in the framework of algebraifolds, forming $f^{-1}$ can make sense for all $f \neq 0$, since the algebraifold containing $f$ may even be a field.
This case seems particularly close to physicists' general intuition.
To see what the resulting algebraifolds can look like, let us consider a simple cosmological spacetime as an example.
In a textbook treatment like Wald's~\cite[Chapter~5]{wald}, this would usually be given by the smooth manifold $M \coloneqq \R_{>0} \times \R^3$ with coordinates denoted $(t,x,y,z)$ and a metric of the form
\begin{equation}
	\label{friedmann}
	g = dt \otimes dt - a^2 (dx \otimes dx + dy \otimes dy + dz \otimes dz),
\end{equation}
where $a \in C^\infty(M)$ is a function subject to certain differential equations.\footnote{Namely having vanishing derivative in the spatial directions $x,y,z$, and satisfying the \emph{Friedmann equations} in the $t$ direction.}
Assuming a flat universe containing pressureless matter and cosmological constant $\Lambda = 0$, solving these equations results in
\begin{equation}
	\label{friedmann_a}
	a = C t^{2/3}
\end{equation}
for an integration constant $C > 0$~\cite[Table~5.1]{wald}.
Throughout the following, we set $C = 1$ for simplicity.

To get an algebraifold that is also a field and still describes this spacetime, it is natural to consider the field of fractions of the $\R$-algebra
\begin{equation}
	\label{friedmann_algebra}
	\R[a,t,x,y,z] / (a^3 - t^2),
\end{equation}
which is an integral domain since the polynomial $a^3 - t^2$ is irreducible.
Defining $\ZLA$ to be the field of fractions of~\eqref{friedmann_algebra} indeed produces a finitely generated field extension of $k \coloneqq \R$, on which~\eqref{friedmann} defines a metric in our sense.\footnote{\label{cuspidal_curve} It may be worth noting that this field of fractions is isomorphic to the rational function field $\R(s,x,y,z)$ involving a new variable $s$. While this will be obvious to any algebraic geometer, it is challenging to find a reference which would show this explicitly in a way which also applies over $\R$, so let us give a sketch: there is a homomorphism
\[
	\R[a,t,x,y,z] \to \R[s,x,y,z] 
\]
given by $t \mapsto s^3$ and $a \mapsto s^2$ and sending the other variables to themselves.
Its kernel is exactly the principal ideal generated by $a^3 - t^2$, and this induces an isomorphism between $\R[a,t,x,y,z] / (a^3 - t^2)$ and the subalgebra $\R[s^2,s^3,x,y,z] \subseteq \R[s,x,y,z]$.
The proof is then completed by noting that $s = \frac{s^3}{s^2}$ belongs to the field of fractions of this subalgebra.
}

\begin{rem}
	\begin{enumerate}
		\item At the present stage, we do not want to claim that this algebraifold is an adequate mathematical model of the universe from the physical point of view.
			Whether this is the case---of course under the highly idealized assumptions that standardly lead to~\eqref{friedmann} and~\eqref{friedmann_a}---remains to be seen.

			For one thing, a physical model will need to have empirical content through making predictions about observations.
			One way to go in this direction could be to consider measurements of an observer moving on a geodesic (\Cref{geodesics}) in spacetime.
			Deriving predictions based on this may first require developing additional aspects of differential geometry first, such as the geodesic deviation equation.
		\item We have made no use of the fact that our ground field is $\R$.
			So at least from the purely mathematical perspective, we still have a perfectly valid Einstein algebra over \emph{any} field $k$ in place of $\R$, where the required separability holds by the argument of \Cref{cuspidal_curve}.
		\item The reader may wonder whether~\eqref{friedmann_algebra} itself would serve as an Einstein $\R$-algebra with respect to the metric~\eqref{friedmann}.
			Since the variety defined by the equation $a^3 - t^2 = 0$ is singular, the Zariski--Lipman conjecture (\Cref{algebraic_varieties}) suggests that this is not the case.
			Intuitively, the singularity here is a manifestation of the big bang singularity at $t = 0$.
			Going to the field of fractions in particular enforces the invertibility of $t$.
			This is the algebraic analogue of the fact that the big bang singularity is not part of the spacetime manifold in standard general relativity.
	\end{enumerate}
\end{rem}

\subsection{Parametrized spacetimes}

Also the construction of algebraifolds from smooth submersions (\Cref{parametric}) may be of some interest for general relativity.
The idea is that many solutions of the Einstein field equation have free parameters, and the resulting parameter space can be taken to be the base manifold $N$, while the thus parametrized spacetimes are the fibres of the submersion $f$.

For example, the \emph{Kerr metric} is a family of vacuum solutions parametrized by the mass $m$ and the angular momentum $j$ of a black hole.
The set of possible values of $m$ and $j$ is a subset of $\R^2$ which will be our parameter manifold $N$.
In natural units, this set is given by
\[
	N = \{ (m,j) \in \R^2 \mid m > 0, \; \lvert j \rvert < m^2 \},
\]
where we use strict inequalities mainly to ensure that $N$ is a manifold without boundary.
Since the topology of the Kerr spacetime is the same for all values of $m$ and $j$, we can define the smooth submersion $f : M \to N$ as the trivial bundle with the topology of Kerr spacetime on each fibre.
Then $C^\infty(M)$ becomes an algebraifold over $C^\infty(N)$.
Since a metric on this algebraifold in our sense is given by a smoothly varying family of metrics on the fibres by \Cref{tensors_examples}\Cref{parametric_tensors}, and the Kerr metric varies smoothly in $m$ and $j$, we indeed obtain a metric on $C^\infty(M)$ as an algebraifold over $C^\infty(N)$.

\begin{rem}
	\begin{enumerate}
		\item\label{metaphysics_comment} A bold metaphysical interpretation of this Einstein algebra might be that the total space $M$ is the \emph{actual} spacetime, but that observers living in $M$ will not notice this since different fibres do not interact.
			This is because the module of derivations $\D_{C^\infty(M)}$ only consists of vertical vector fields, which makes it is impossible for particles, field excitations or observers to ``leave'' their fibre.
		\item Provided that our parametrization construction can be generalized such that the base $N$ is allowed to be suitably infinite-dimensional, it is conceivable that one can even take it to be a suitably defined space of \emph{all} solutions to the Einstein field equation.
			In this case, the totality of all possible spacetimes consistent with general relativity would be \emph{one single} algebraifold.
			If this works, then it is natural to expect that this algebraifold would be a \emph{universal} solution to the Einstein field equations: a terminal object in a suitably defined category of Einstein algebras.
	\end{enumerate}
\end{rem}

\appendix
\section{Finitely generated projective modules}
\label{fgp}

Throughout this appendix, $\ZLA$ denotes an arbitrary commutative ring.
Let us begin by recalling some basic definitions.

\begin{defn}
	An $\ZLA$-module $\Mod$ is:
	\begin{enumerate}
		\item \newterm{finitely generated} if there is a finite subset $\mathcal{F} \subseteq \Mod$ such that every element of $\Mod$ can be written as a (finite) $\ZLA$-linear combination of elements of $\mathcal{F}$. 
		\item \newterm{projective} if there is an $\ZLA$-module $\Nod$ such that $\Mod \oplus \Nod$ is a free module.
	\end{enumerate}	
\end{defn}

Although this is not immediate from the definition, these two properties interact in such a nice way that their conjunction is especially important, and we thus abbreviate
\[
	\text{\newterm{fgp}} = \text{finitely generated projective}.
\]
The following alternative characterization explains why fgp modules are important, and we make frequent use of properties~\ref{dual_basis_lemma} and~\ref{dualizable_module}.

\begin{thm}
	\label{fgp_dual}
	For an $\ZLA$-module $\Mod$, the following are equivalent:
	\begin{enumerate}
		\item\label{is_fgp} $\Mod$ is fgp.
		\item\label{dual_basis_lemma} \newterm{Dual basis lemma:} There are elements $u_1, \dots, u_n \in \Mod$ and\footnote{Here, $\Mod^* \coloneqq \Modules{\ZLA}(\Mod, \ZLA)$ denotes the dual module.} $\xi_1, \ldots, \xi_n \in \Mod^*$ such that
			\beq
				\label{duality}
				\sum_{i=1}^n \xi_i(x) \, u_i = x
			\eeq
			for all $x \in \Mod$.
		\item\label{fgp_retract}
			There exist $n \in \N$ and module homomorphisms $u : \ZLA^n \to \Mod$ and $\xi : \Mod \to \ZLA^n$ such that $\xi u = \id_{\ZLA^n}$.
		\item\label{dualizable_module} $\Mod$ is a dualizable object in the symmetric monoidal category of $\ZLA$-modules: there is an object $\Mod^*$ together with $\ZLA$-linear map
			\beq
				\delta \: : \: \ZLA \longrightarrow \Mod \otimes_\ZLA \Mod^*, \qquad \varepsilon \: : \: \Mod^* \otimes_\ZLA \Mod \longrightarrow \ZLA
			\eeq
			such that the triangle identities
			\beq
				\label{triangle_ids}
				\begin{tikzcd}[sep=3.1pc]
					\Mod \ar[equals]{dr} \ar{r}{\delta \,\otimes\, \id_\Mod}	& \Mod \otimes_\ZLA \Mod^* \otimes_\ZLA \Mod \ar{d}{\id_\Mod \,\otimes\, \varepsilon}	& \Mod^* \ar[equals]{dr} \ar{r}{\id_{\Mod^*} \,\otimes\, \delta}	& \Mod^* \otimes_\ZLA \Mod \otimes_\ZLA \Mod^*  \ar{d}{\varepsilon \,\otimes\, \id_{\Mod^*}}	\\
												& \Mod	&	& \Mod^*
				\end{tikzcd}
			\eeq
			hold.
	\end{enumerate}
\end{thm}

\begin{proof}
	See e.g.~\cite[Remark~2.11]{lam} for the equivalence of~\ref{is_fgp} and~\ref{dual_basis_lemma}.
	Condition~\ref{fgp_retract} is a straightforward rephrasing of~\ref{dual_basis_lemma}.
	Finally, see e.g.~\cite[Example~1.4]{DP} or~\cite[Example~3.2]{PS} for the equivalence with~\ref{dualizable_module}.
\end{proof}

\begin{rem}
	\label{fgp_rem}
	\begin{enumerate}
		\item Perhaps in contrast to what the phrasing ``dual basis lemma'' suggests, the elements $u_1, \ldots, u_n$ generate $\Mod$ but do not need to form a basis (since $\Mod$ does not even need to be free), and similarly the $\xi_1, \ldots, \xi_n$ generate $\Mod^*$ without necessarily forming a basis.
		\item Since property~\ref{dualizable_module} is invariant under exchanging $\Mod$ and $\Mod^*$, we know that $\Mod^*$ is automatically fgp as well, and in particular the dual basis lemma holds for it: we have 
			\begin{equation}
				\label{duality2}
				\sum_{i=1}^n \eta(u_i) \, \xi_i = \eta.
			\end{equation}
			for all $\eta \in \Mod^*$.
		\item The object $\Mod^*$ in~\ref{dualizable_module} not not be assumed to be (isomorphic to) the dual module $\Modules{\ZLA}(\Mod, \ZLA)$ a priori. 
			Indeed the proofs referenced above show that if $\Mod^*$ is \emph{any} $\ZLA$-module with $\delta$ and $\varepsilon$ satisfying the triangle identities, there is a canonical isomorphism $\Mod^* \cong \Modules{\ZLA}(\Mod, \ZLA)$ such that $\varepsilon$ turns into the standard evaluation map
			\begin{align}
			\begin{split}
				\label{ev_map}
				\Mod^* \otimes_\ZLA \Mod	& \longrightarrow \ZLA,	\\
				f \otimes x			& \longmapsto f(x).
			\end{split}
			\end{align}
			In other words, one can assume without loss of generality that $\Mod^* = \Modules{\ZLA}(\Mod, \ZLA)$ and that $\varepsilon$ is given by the evaluation map.
			This is also why $\delta$ is often called the \emph{coevaluation map}.
			With $\varepsilon$ given by~\eqref{ev_map}, one can show that it is given by
			\[
				\delta = \sum_{i=1}^n u_i \otimes \xi_i, 
			\]
			with $u_i \in \Mod$ and $\xi_i \in \Mod^*$ as in~\ref{dual_basis_lemma}.
		\item The homomorphisms $\delta$ and $\varepsilon$ induce a natural bijection
			\beq
				\label{fgp_hom}
				\begin{tikzcd}[row sep=4pt]
					\Modules{\ZLA}(\mathcal{R} \otimes \Mod, \mathcal{S}) \ar[draw=none,"\displaystyle{\cong}" description]{r}	& \Modules{\ZLA}(\mathcal{R}, \mathcal{S} \otimes \Mod^*)	\\
					f \ar[mapsto]{r}	& (f \otimes \id_{\Mod^*}) \circ (\id_{\mathcal{R}} \otimes \delta)	\\
					(\id_{\mathcal{S}} \otimes \varepsilon) \circ (g \otimes \id_{\Mod^*})			& g \ar[mapsto]{l}
				\end{tikzcd}
			\eeq
			for all $\ZLA$-modules $\mathcal{R}$ and $\mathcal{S}$.
			In terms of dual bases as above, the counterpart of $f : \mathcal{R} \otimes \Mod \to \mathcal{S}$ is given by
			\begin{align}
			\begin{split}
				\label{dual_basis_adjunct}
				\mathcal{R} & \longrightarrow \mathcal{S} \otimes \Mod^* \\
				r & \longmapsto \sum_{i=1}^n f(r \otimes u_i) \otimes \xi_i,
			\end{split}
			\end{align}
			while the counterpart of $g : \mathcal{R} \to \mathcal{S} \otimes \Mod^*$ is
			\begin{align}
			\begin{split}
				\label{dual_basis_adjunct_inv}
				\mathcal{R} \otimes \Mod & \longrightarrow \mathcal{S} \\
				r \otimes x & \longmapsto g_{(2)}(r)(x) \cdot g_{(1)}(r),
			\end{split}
			\end{align}
			where we use sumless Sweedler notation for $g$.
			The fact that these two constructions are each other's inverses is straightforward to prove from~\eqref{duality}.
		\item\label{string_diagrams} This bijection is most easily understood in terms of the \emph{graphical calculus} of symmetric monoidal categories~\cite{selinger}, in which the two maps are given by
			
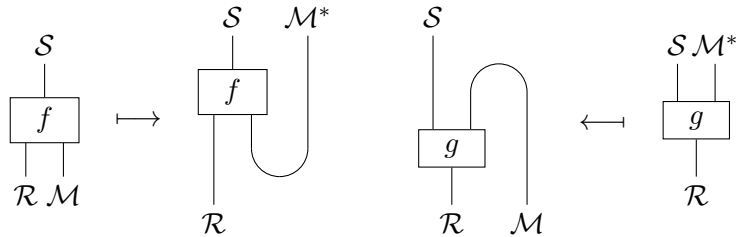
\[
				\begin{tikzpicture}
	\begin{pgfonlayer}{nodelayer}
		\node [style=box] (0) at (-2.5, 0) {$\;\; f \;\;$};
		\node [style=none] (1) at (-3, -1.5) {};
		\node [style=none] (2) at (-3, 0) {};
		\node [style=none] (3) at (-2, 0) {};
		\node [style=none] (4) at (-2.5, 1.5) {};
		\node [style=none] (5) at (-2.5, 2) {$\mathcal{S}$};
		\node [style=none] (6) at (-3, -2) {$\mathcal{R}$};
		\node [style=none] (7) at (-2, -2) {$\Mod$};
		\node [style=none] (8) at (-2, -1.5) {};
		\node [style=none] (9) at (0, 0) {$\longmapsto$};
		\node [style=box] (13) at (2.5, 0.75) {$\;\; f \;\;$};
		\node [style=none] (14) at (2, -2.25) {};
		\node [style=none] (15) at (2, 0.75) {};
		\node [style=none] (16) at (3, 0.75) {};
		\node [style=none] (17) at (2.5, 2.25) {};
		\node [style=none] (18) at (2.5, 2.75) {$\mathcal{S}$};
		\node [style=none] (19) at (2, -2.75) {$\mathcal{R}$};
		\node [style=none] (21) at (3, -0.75) {};
		\node [style=none] (25) at (4.5, 2.75) {$\Mod^*$};
		\node [style=none] (26) at (4.5, 2.25) {};
		\node [style=none] (27) at (4.5, -0.75) {};
		\node [style=none] (29) at (3.75, -1.5) {};
	\end{pgfonlayer}
	\begin{pgfonlayer}{edgelayer}
		\draw [style=edge, in=0, out=-90] (27.center) to (29.center);
		\draw [style=edge, in=-90, out=180] (29.center) to (21.center);
		\draw [style=edge] (2.center) to (1.center);
		\draw [style=edge] (3.center) to (8.center);
		\draw [style=edge] (4.center) to (0);
		\draw [style=edge] (17.center) to (13);
		\draw [style=edge] (15.center) to (14.center);
		\draw [style=edge] (16.center) to (21.center);
		\draw [style=edge] (26.center) to (27.center);
	\end{pgfonlayer}
\end{tikzpicture}

				\qquad\quad
				\begin{tikzpicture}
	\begin{pgfonlayer}{nodelayer}
		\node [style=box] (0) at (2.5, 0) {$\;\; g \;\;$};
		\node [style=none] (1) at (2, 1.5) {};
		\node [style=none] (2) at (2, 0) {};
		\node [style=none] (3) at (3, 0) {};
		\node [style=none] (4) at (2.5, -1.5) {};
		\node [style=none] (5) at (2.5, -2) {$\mathcal{R}$};
		\node [style=none] (6) at (2, 2) {$\mathcal{S}$};
		\node [style=none] (7) at (3, 2) {$\Mod^*$};
		\node [style=none] (8) at (3, 1.5) {};
		\node [style=none] (9) at (0, 0) {$\longmapsfrom$};
		\node [style=box] (10) at (-4, -0.75) {$\;\; g \;\;$};
		\node [style=none] (11) at (-4.5, 2.25) {};
		\node [style=none] (12) at (-4.5, -0.75) {};
		\node [style=none] (13) at (-3.5, -0.75) {};
		\node [style=none] (14) at (-4, -2.25) {};
		\node [style=none] (15) at (-4, -2.75) {$\mathcal{R}$};
		\node [style=none] (16) at (-4.5, 2.75) {$\mathcal{S}$};
		\node [style=none] (17) at (-3.5, 0.75) {};
		\node [style=none] (18) at (-2, -2.75) {$\Mod$};
		\node [style=none] (19) at (-2, -2.25) {};
		\node [style=none] (20) at (-2, 0.75) {};
		\node [style=none] (21) at (-2.75, 1.5) {};
	\end{pgfonlayer}
	\begin{pgfonlayer}{edgelayer}
		\draw [style=edge, in=0, out=90] (20.center) to (21.center);
		\draw [style=edge, in=90, out=180] (21.center) to (17.center);
		\draw [style=edge] (2.center) to (1.center);
		\draw [style=edge] (3.center) to (8.center);
		\draw [style=edge] (4.center) to (0);
		\draw [style=edge] (14.center) to (10);
		\draw [style=edge] (12.center) to (11.center);
		\draw [style=edge] (13.center) to (17.center);
		\draw [style=edge] (19.center) to (20.center);
	\end{pgfonlayer}
\end{tikzpicture}

			\]
			and the triangle identities~\eqref{triangle_ids} take the very intuitive form
			\[
				\begin{tikzpicture}
	\begin{pgfonlayer}{nodelayer}
		\node [style=none] (1) at (-8, 0) {};
		\node [style=none] (2) at (-6.5, -2.5) {$\Mod$};
		\node [style=none] (3) at (-6.5, -2) {};
		\node [style=none] (4) at (-6.5, 0) {};
		\node [style=none] (5) at (-7.25, 0.75) {};
		\node [style=none] (7) at (-9.5, 0) {};
		\node [style=none] (8) at (-9.5, 2.5) {$\Mod$};
		\node [style=none] (10) at (-8, 0) {};
		\node [style=none] (11) at (-8.75, -0.75) {};
		\node [style=none] (12) at (-9.5, 2) {};
		\node [style=none] (13) at (-4.5, 0) {$=$};
		\node [style=none] (14) at (-2.5, 2) {};
		\node [style=none] (15) at (-2.5, -2) {};
		\node [style=none] (16) at (-2.5, -2.5) {$\Mod$};
		\node [style=none] (17) at (-2.5, 2.5) {$\Mod$};
		\node [style=none] (18) at (4, 0) {};
		\node [style=none] (19) at (2.5, -2.5) {$\Mod^*$};
		\node [style=none] (20) at (2.5, -2) {};
		\node [style=none] (21) at (2.5, 0) {};
		\node [style=none] (22) at (3.25, 0.75) {};
		\node [style=none] (23) at (5.5, 0) {};
		\node [style=none] (24) at (5.5, 2.5) {$\Mod^*$};
		\node [style=none] (25) at (4, 0) {};
		\node [style=none] (26) at (4.75, -0.75) {};
		\node [style=none] (27) at (5.5, 2) {};
		\node [style=none] (28) at (7.5, 0) {$=$};
		\node [style=none] (29) at (9.5, 2) {};
		\node [style=none] (30) at (9.5, -2) {};
		\node [style=none] (31) at (9.5, -2.5) {$\Mod^*$};
		\node [style=none] (32) at (9.5, 2.5) {$\Mod^*$};
	\end{pgfonlayer}
	\begin{pgfonlayer}{edgelayer}
		\draw [style=edge, in=0, out=90] (4.center) to (5.center);
		\draw [style=edge, in=90, out=180] (5.center) to (1.center);
		\draw [style=edge] (3.center) to (4.center);
		\draw [style=edge, in=0, out=-90] (10.center) to (11.center);
		\draw [style=edge, in=-90, out=180] (11.center) to (7.center);
		\draw [style=edge] (12.center) to (7.center);
		\draw [style=edge] (14.center) to (15.center);
		\draw [style=edge, in=180, out=90] (21.center) to (22.center);
		\draw [style=edge, in=90, out=0] (22.center) to (18.center);
		\draw [style=edge] (20.center) to (21.center);
		\draw [style=edge, in=180, out=-90] (25.center) to (26.center);
		\draw [style=edge, in=-90, out=0] (26.center) to (23.center);
		\draw [style=edge] (27.center) to (23.center);
		\draw [style=edge] (29.center) to (30.center);
	\end{pgfonlayer}
\end{tikzpicture}

			\]
			which are often called the \emph{zig-zag identities}.
	\end{enumerate}
\end{rem}

Let us now explain the geometrical significance of fgp modules.
For a smooth manifold $M$, we write $\Vect{M}$ for the category of smooth real vector bundles over $M$, which is a symmetric monoidal category with respect to the usual tensor product of vector bundles.
For a ring $\ZLA$, we write $\fgpModcat{\ZLA}$ for the category of fgp $\ZLA$-modules, which is a symmetric monoidal category with respect to the usual algebraic tensor product of $\ZLA$-modules.
Then the following result is the \newterm{smooth Serre--Swan theorem}~\cite[Theorem~11.32]{nestruev}, extended to a symmetric monoidal equivalence of symmetric monoidal categories.

\begin{thm}
	\label{SSS}
	If $M$ is a smooth manifold with finitely many connected components, then the functor
	\[
		\Gamma^\infty \: : \: \Vect{M} \longrightarrow \fgpModcat{C^\infty(M)}
	\]
	mapping every vector bundle $\pi_E : E \to M$ to the $C^\infty(M)$-module $\Gamma^\infty(E)$ of smooth sections of $E$ is a symmetric monoidal equivalence
\end{thm}

Note that this functor is indeed covariant: although equivalences of categories between geometrical objects and algebraic objects are often contravariant, this is not the case here, intuitively because vector bundles already are of an algebraic nature.

\begin{proof}
	That this functor indeed lands in fgp modules and is an equivalence of categories for connected $M$ is~\cite[Theorem~11.32]{nestruev} as cited above.
	The same statement for finitely many connected components is directly implied, as both categories can then be described as product categories: if $M$ has connected components $M_1, \ldots, M_n$, then each vector bundle decomposes uniquely into a direct sum of vector bundles supported on each component, and this implies that the category is naturally equivalence to a product category,
	\[
		\Vect{M} \cong \prod_{i=1}^n \Vect{M_i}.
	\]
	The same applies at the level of fgp modules,
	\[
		\fgpModcat{C^\infty(M)} \cong \prod_{i=1}^n \fgpModcat{C^\infty(M_i)}.
	\]
	Furthermore, the functor $\Gamma^\infty$ acts as the corresponding equivalence $\Vect{M_i} \cong \fgpModcat{C^\infty(M_i)}$ on each factor, and therefore is an equivalence itself.

	To show that $\Gamma^\infty$ is a symmetric monoidal equivalence, note that for any two vector bundles
	\[
		\pi_E : E \to M \quad \textrm{and} \quad \pi_F : F \to M,
	\]
	a pair of sections $\xi \in \Gamma^\infty(M)$ and $\eta \in \Gamma^\infty(F)$ induces a section $\xi \otimes \eta \in \Gamma^\infty(E \otimes F)$ by tensoring them pointwise. 
	This construction is an $\ZLA$-bilinear map, which induces a module homomorphism
	\begin{equation}
		\label{tensor_map}
		\Gamma^\infty(E) \otimes_{C^\infty(M)} \Gamma^\infty(F) \longrightarrow \Gamma^\infty(E \otimes F)
	\end{equation}
	natural in $E$ and $F$.
	In combination with the trivial isomorphism $\Gamma^\infty(M) \cong C^\infty(M)$, which amounts to $\Gamma^\infty$ preserving the monoidal unit, it is straightforward to verify the relevant coherences which show that $\Gamma^\infty$ is a lax symmetric monoidal functor.

	It remains to be proven that the structure maps~\eqref{tensor_map} are isomorphisms of $C^\infty(M)$-modules.
	This is~\cite[Theorem~7.5.5]{conlon}, but we sketch the proof here as representative of a standard argument involving fgp modules.
	The fact that~\eqref{tensor_map} is an isomorphism when $E$ or $F$ is the trivial one-dimensional vector bundle $M$ is clear.
	Since both sides are additive in $E$ and $F$, the naturality implies that the map is an isomorphism also for all trivial vector bundles $E$ and $F$ of finite rank.
	The general case now follows by another application of additivity and naturality, using the fact that every vector bundle is a direct summand of a trivial one.
\end{proof}

\begin{rem}
	It seems plausible that some version of \Cref{SSS} is still correct even with infinitely many connected components, but care needs to be taken with the definition of vector bundle: in order for the equivalence to hold, one can clearly not expect the bundle dimension to be constant across the components of $M$, but it still needs to be bounded in order for the associated module to be finitely generated.
\end{rem}

\section{Extension and restriction of scalars}
\label{sec_extension}

Given an arbitrary ring homomorphism $\varphi : \ZLA \to \ZLB$, it is a standard fact that there is an adjunction between categories of modules like this:
\begin{equation}
	\label{adjunction_scalars}
	\begin{tikzcd}
		\Modcat{\ZLA} \ar[bend left]{r}{\varphi_*} \ar[phantom]{r}{\perp}	& \Modcat{\ZLB} \ar[bend left]{l}{\varphi^*}
	\end{tikzcd}
\end{equation}
The functor $\varphi^*$ is called \newterm{restriction of scalars}, and is simply given by considering every $\ZLB$-module as an $\ZLA$-module via $\varphi$.
The functor $\varphi_*$ is called \newterm{extension of scalars} and can be constructed as
\begin{equation}
	\label{extension}
	\varphi_*(\Mod) \coloneqq \ZLB \otimes_\ZLA \Mod,
\end{equation}
where $\ZLB$ is considered as an $\ZLA$-module via $\varphi$, and the result is a $\ZLB$-module with respect to $\ZLB$ acting by multiplication from the left.
The functoriality in $\Mod$ is obvious.
The adjunction then amounts to the hom-set bijection
\[
	\Modcat{\ZLA}(\ZLB \otimes_\ZLA \Mod, \Nod) \cong \Modcat{\ZLB}(\Mod, \varphi^*(\Nod)) \qquad \forall \Mod \in \Modcat{\ZLA}, \: \Nod \in \Modcat{\ZLB}
\]
which follows straightforwardly from the universal property of the tensor product.

Let us argue now that the functor $\varphi_*$ is especially well-behaved on the full subcategory of fgp modules $\fgpModcat{\ZLA} \subseteq \Modcat{\ZLA}$.

\begin{prop}
	\label{extension_fgp}
	The extension of scalars functor
	\[
		\varphi_* : \fgpModcat{\ZLA} \longrightarrow \fgpModcat{\ZLB}
	\]
	is strong symmetric monoidal.
\end{prop}

\begin{proof}
	First of all, it is straightforward to see from~\eqref{extension} that if $\Mod$ is a direct summand of a finitely generated free module, then so is $\varphi_*(\Mod)$.
	This shows that $\varphi_*$ indeed maps $\fgpModcat{\ZLA}$ to $\fgpModcat{\ZLB}$.
	Now for given fgp $\ZLA$-modules $\Mod$ and $\Nod$, we can argue as in the proof of \Cref{SSS} in order to construct a natural isomorphism
	\[
		\varphi_*(\Mod) \otimes_\ZLB \varphi_*(\Nod) \cong \varphi_*(\Mod \otimes_{\ZLA} \Nod)
	\]
	satisfying the relevant coherences, where the monoidal unit isomorphism $\varphi_*(\ZLA) \cong \ZLB$ is trivial.
\end{proof}

\begin{rem}
	\label{extension_dual}
	It follows from~\Cref{extension_fgp} that the extension of scalars commutes with dualization: there is an isomorphism
	\[
		\varphi_*(\Mod^*) \cong \varphi_*(\Mod)^*
	\]
	natural in fgp $\ZLA$-modules $\Mod$.
	Given dual bases of $\Mod$ and $\Mod^*$ as in~\eqref{duality}, one can also see directly that the elements
	\begin{equation}
		\label{extension_dual_eq}
		1_\ZLB \otimes u_1, \ldots, 1_\ZLB \otimes u_n \in \varphi_*(\Mod), \qquad
		1_\ZLB \otimes \xi_1, \ldots, 1_\ZLB \otimes \xi_n \in \varphi_*(\Mod^*)
	\end{equation}
	form dual bases of $\varphi_*(\Mod)$ and $\varphi_*(\Mod^*)$, since the relevant relation~\eqref{duality} is straightforward to verify.
\end{rem}

The geometric significance of the extension of scalars is that it corresponds to the pullback of vector bundles along a smooth map, in the following sense.

\begin{prop}
	\label{pullback}
	Let $M$ and $N$ be smooth manifolds with finitely many connected components and $f : M \to N$ a smooth map.
	Then the diagram
	\[
		\begin{tikzcd}
			\Vect{N} \ar{r}{f^*} \ar[swap]{d}{\cong} & \Vect{M} \ar{d}{\cong}	\\	
			\fgpModcat{C^\infty(N)} \ar{r}{C^\infty(f)_*}	& \fgpModcat{C^\infty(M)}
		\end{tikzcd}
	\]
	commutes up to natural isomorphism, where the vertical arrows are the equivalences of \Cref{SSS}.
\end{prop}

\begin{proof}
	This follows again by additivity of the functors involved and the fact that the diagram commutes up to isomorphism on the trivial vector bundle of rank one on $N$.
\end{proof}

\begin{rem}
	\label{pullback_mfd}
	In particular, if $\Mod$ is any fgp $C^\infty(N)$-module corresponding to sections of a vector bundle $\pi_E : E \to N$, then the pullback bundle $f^* E \to M$ has module of sections given by
	\[
		C^\infty(f)_*(\Mod) = C^\infty(M) \otimes_{C^\infty(N)} \Mod.
	\]
	For any section $x \in \Mod$, we furthermore have an associated element
	\[
		1_{C^\infty(M)} \otimes x \: \in \: C^\infty(f)_*(\Mod),
	\]
	which corresponds to the section of $f^* E$ given by pulling back the original section along $f$ in the obvious way.
	This can once again by the seen by the standard additivity and naturality argument, and we use this fact several times in the main text.
\end{rem}

\printbibliography

\end{document}